\documentclass[10pt]{amsart}

\usepackage{color}
\usepackage{graphicx, tikz}
\usepackage{xcolor}
\usepackage{cancel}
\usepackage{amssymb,amsfonts,amsthm,amsmath,calligra, enumerate, stmaryrd}
\usepackage{extarrows}
\usepackage[utf8]{inputenc}
\usepackage{accents}
\usepackage{slashed}
\usepackage{yfonts}
\usepackage{mathrsfs,pifont}
\usepackage{mathtools}

\usetikzlibrary{arrows.meta}

\usepackage[bookmarksnumbered=true, bookmarksopen=true]{hyperref}
\usepackage[all]{xy}

\usepackage[margin=1in]{geometry}
\usepackage{amsmath,amssymb}
\usepackage{xcolor}
\usepackage{tikz}
\usetikzlibrary{arrows.meta}

\theoremstyle{definition}

\newtheorem{Theorem}{Theorem}[section]
\newtheorem{Proposition}[Theorem]{Proposition}
\newtheorem{Lemma}[Theorem]{Lemma}
\newtheorem{Corollary}[Theorem]{Corollary}
\newtheorem{Definition}[Theorem]{Definition}

\newtheorem{Remark}[Theorem]{Remark}
\newtheorem{Example}[Theorem]{Example}

\numberwithin{equation}{section}

\def\ben{\begin{eqnarray*}}
\def\een{\end{eqnarray*}}

\newcommand{\C}{\mathbb{C}}
\newcommand{\GL}{\mathrm{GL}}

\newcommand{\Crit}{\mathrm{Crit}}

\newcommand{\Q}{\mathbb{Q}}

\newcommand{\bC}{\mathbb{C}}

\newcommand{\bZ}{\mathbb{Z}}

\newcommand{\bv}{\mathbf{v}}

\newcommand{\bd}{\mathbf{d}}

\newcommand{\bPsi}{\mathbf{\Psi}}

\newcommand{\cH}{\mathcal{H}}

\newcommand{\cM}{\mathcal{M}}

\newcommand{\fX}{\mathfrak{X}}

\newcommand{\sA}{\mathsf{A}}

\newcommand{\sG}{\mathsf{G}}

\newcommand{\sT}{\mathsf{T}}

\newcommand{\sW}{\mathsf{W}}
\newcommand{\sw}{\mathsf{w}}

\newcommand{\End}{\operatorname{End}}

\newcommand{\Frac}{\operatorname{Frac}}

\newcommand{\Hilb}{\operatorname{Hilb}}
\newcommand{\Hom}{\operatorname{Hom}}

\newcommand{\pt}{\operatorname{pt}}

\newcommand{\rk}{\operatorname{rk}}

\newcommand{\Spec}{\operatorname{Spec}}

\newcommand{\tr}{\operatorname{tr}}

\newcommand{\id}{\operatorname{id}}

\newcommand{\In}{\operatorname{in}}
\newcommand{\Out}{\operatorname{out}}

\makeatletter
\newcommand{\ostar}{\mathbin{\mathpalette\make@circled\star}}
\newcommand{\make@circled}[2]{%
  \ooalign{$\m@th#1\smallbigcirc{#1}$\cr\hidewidth$\m@th#1#2$\hidewidth\cr}%
}
\newcommand{\smallbigcirc}[1]{%
  \vcenter{\hbox{\scalebox{0.77778}{$\m@th#1\bigcirc$}}}%
}
\makeatother

\title{Injectivity of specialization maps in critical cohomology}

\author{Yalong Cao}
\address{Morningside Center of Mathematics, Institute of Mathematics \& State Key Laboratory of Mathematical Sciences, Academy of Mathematics and Systems Sciences, Chinese Academy of Sciences, Beijing, China}
\email{yalongcao@amss.ac.cn}

\subjclass[2020]{
Primary
14N35,  
17B37. 
Secondary 
16G20, 
14C17} 
\keywords{Critical cohomology, specialization maps, quivers with potentials}

\begin{document}

\begin{abstract}
Specialization maps are natural maps from equivariant vanishing cycle cohomology to the equivariant cohomology of the ambient space.
We study their injectivity problem for moduli stacks of representations of quivers with potentials, including those
associated with generalized conifolds and finite subgroups of $SO(3,\mathbb{C})$. 
Applications to the irreducibility of modules of Reshetikhin type shifted Yangians are also given.  
\end{abstract}

\maketitle

\setcounter{tocdepth}{1}
\tableofcontents




\section{Introduction}

\subsection{Background}

Critical cohomology (vanishing cycle cohomology) and its $K$-theoretic analogue are fundamental objects in the theory of perverse sheaves and singularities.  
They have been heavily used in the current research on geometric representation theory,
for example, Kontsevich-Soibelman critical cohomological Hall algebras, e.g.~\cite{KS, Dav1, DM, RSYZ, BR} and its $K$-theoretic analogue \cite{P}, geometric realization of Kirillov-Reshetikhin modules \cite{VV1, VV2, VV3}, 
geometric $R$-matrices given by critical stable envelopes \cite{COZZ1,COZZ2, COZZ3, COZZ4}, Coulomb branch algebras \cite{Nak, BFN, BT} and many others.
They also play an important role in enumerative geometry, where they serve as the state spaces in Gromov-Witten type theory of critical loci (including the theory of gauged linear sigma models)
e.g.~\cite{FJR, CFGKS, FK, AL, CZ, CTZ},
also in relative Donaldson-Thomas theory of logarithmic Calabi-Yau 4-folds, e.g. \cite{CZZ, CTZ, Liu}. 

While critical cohomology encodes rich singularity-theoretic information, this very richness makes it technically challenging to study and compute. 
For example, even constructing explicit classes in critical cohomology is generally difficult. The usual approach of taking Chern classes of vector bundles on the ambient space does not directly apply: these Chern classes naturally act on critical cohomology but do not, in general, define elements of it.

Several techniques are available for studying critical cohomology, including dimensional reduction, which under suitable assumptions, reduces its computation to that of the ordinary cohomology of an algebraic variety \cite{Dav1}. 
Another way to relate critical cohomology to ordinary cohomology is by the specialization map \cite{COZZ1}. 
\begin{Proposition}(\cite[Props 7.16--7.18, Ex.~7.19]{COZZ1})
Let $X$ be a smooth quiver variety\footnote{They are GIT quotients of vector spaces by reductive groups with no relations imposed. 
Specialization maps can be defined for more general setting, see \cite[Def.~7.14]{COZZ1}. } 
with a $\sT$-action, and let $\sw$ be a $\sT$-invariant regular function on $X$. Then there is a naturally defined map 
(called \textit{specialization map})\footnote{There is a more general specialization map relating different critical cohomology, see \cite[Def.~7.14]{COZZ1}.}:
$$\mathsf{sp}\colon H^\sT(X,\sw)\to H^\sT(X). $$
The map is functorial with respect to l.c.i.\,pullbacks, proper pushforward and compatible with canonical maps and stable envelopes.   
\end{Proposition}

Knowing the \textit{injectivity} of specialization maps is useful from several perspectives. For example, one can study when a Chern class of 
a vector bundle admits a (unique) lift to the critical cohomology via the map. In representation theory, since specialization maps are 
compatible with stable envelopes, they are module maps of corresponding (shifted) quantum groups, injectivity can have  purely algebraic applications 
for those modules (see Theorem \ref{thm on irr intro}). 
This is also helpful in enumerative geometry, for example, one may reduce the curve counts on critical locus to the curve counts on the ambient spaces (with suitable constraints),  
another example is the following 
formula on quantum multiplication by divisors on critical cohomology, which uses injectivity of specialization maps.

\begin{Theorem}(\cite[Thm.~1.5, Thm.~1.6]{COZZ2})\label{thm quantum mult}
Given a symmetric quiver $Q$ with potential $\sW$, there is an associated Lie superalgebra $\mathfrak{g}_{Q,\sW}$.
Let $\cM(\bv,\bd)$ be a symmetrically framed quiver variety with induced function $\sw=\tr(\sW)$.

Suppose that the specialization map $\mathsf{sp}\colon H^\sT(\cM(\bv,\bd),\sw^{\mathrm{}})\to H^\sT(\cM(\bv,\bd))$ is injective, then the modified quantum multiplication by divisor $\lambda$ on $H^\sT(\cM(\bv,\bd),\sw^{\mathrm{}})[\![z]\!]$ is given by
$$c_1(\lambda)\,\widetilde{\star}\, \cdot =c_1(\lambda)\cup \cdot -\sum_{\alpha>0}\alpha(\lambda)\frac{z^\alpha}{(-1)^{|\alpha|}-z^\alpha}\mathrm{Cas}_\alpha+\cdots \,,$$
where $\mathrm{Cas}_\alpha=\sum_{s} e_{\alpha}^{(s)} e_{-\alpha}^{(s)}$ is the Casimir operator of $\mathfrak{g}_{Q,\sW}$, and dots denote a scalar operator. 

There is a similar formula for the case of anti-dominantly framed quiver varieties with induced functions. 
\end{Theorem}

The injectivity of (localized) specialization maps has been established for quivers with potentials 
associated with affine $\widehat{\mathfrak{gl}}_1$ and any symmetrizable Kac-Moody Lie algebras \cite[Lem.~A.4]{COZZ2}. 
In this paper, we extend those results to more examples, particularly 
those that are connected to moduli spaces of sheaves on (non-compact) Calabi-Yau 3-folds. 

\subsection{Main results}
The first series of examples we study is the \textit{generalized conifolds}:
$$\{xy=z^mw^n\}\subseteq \mathbb{C}^4. $$
The corresponding quivers with potentials are described in \S \ref{sect on QP gene conifold}. 
\begin{Theorem}(Theorems \ref{thm on sp inj gene stack}, \ref{thm on sp of framed gene conifolds})
Let $\fX=[R/\sG]$ be a stack of representations of $(Q,\sW)$ associated with a generalized conifold \S \ref{sect on QP gene conifold}. Then the specialization map 
$$\mathsf{sp}: H^{\sT}(\fX,\sw)\to H^{\sT}(\fX)$$
is injective. Moreover, specialization maps for moduli stacks of framed (stable) representations with the pullback function $\sw$ are also injective.
\end{Theorem}
We first prove the injectivity after localization with respect to the flavor torus and then deduce the injectivity without localization using a purity argument \cite{Dav2}. 
The corresponding injectivity claim for moduli of stable framed  representations follows from the compatibility of specialization map with 
the so-called \textit{interpolation map} (which is a generalization of nonabelian stable envelopes), see Proposition \ref{prop on nonabel stab and sp}. 
 
The second series of examples is associated with \textit{finite subgroups} $G<SO(3,\C)$, which are classified by the $ADE$ type \S \ref{sect on finite subgroup SO(3)}. 
The $A$-type fits into the symmetrizable Kac-Moody case studied in \cite[Lem.~A.4]{COZZ2}. 
Below we treat the $D$-type case, where the associated McKay quiver with cubic potential is presented in \S \ref{sect on QP of polyhedra}. 

\begin{Theorem}(Theorems \ref{prop on inj polyhedral}, \ref{thm on inj D4})
\label{thm intro SO3}
Let $\fX=[R/\sG]$ be a stack of representations of $(Q,\sW)$ associated with a finite subgroup $G<SO(3,\C)$ of $D$-type. 
Then the specialization map 
$$\mathsf{sp}: H^{\sT}(\fX,\sw) \to  H^{\sT}(\fX)$$
is injective after localization with respect to the flavor torus $\sT$-action. 
\end{Theorem}

For type $E$, we show the following non-injectivity in \S \ref{app on type E}. 

\begin{Proposition}(Proposition \ref{prop on type E}, Remark \ref{rmk on type E})
Let $\fX=[R/\sG]$ be a stack of representations of $(Q,\sW)$ associated with a finite subgroup $G<SO(3,\C)$ of $E$-type such that $\dim R> 0$. 
Then the specialization map 
$$\mathsf{sp}: H^{\sT}(\fX,\sw) \to  H^{\sT}(\fX)$$
is \textit{not} injective after $\sT$-localization. 
\end{Proposition}

Besides the immediate application to Theorem \ref{thm quantum mult}, the above injectivity results can also be used to deduce the irreducibility of shifted Yangian modules. 
Recall \cite[\S 4]{COZZ2}, given a symmetric quiver $Q$ with a $\sT$-action, and a $\sT$-invariant potential $\sW$, 
satisfying the hypotheses of \textit{loc.\,cit.}, fix $\mu\in \bZ_{\leqslant 0}^{Q_0}$, 
there is a \textit{Reshetikhin type} $\mu$-\textit{shifted Yangian} $\mathsf Y_\mu(Q,\sW)$ 
constructed by geometric $R$-matrices defined by critical stable envelopes.
The Yangian acts on equivariant critical cohomology of 
$\mu=\bd_{\Out}-\bd_{\In}$ antidominantly framed quiver varieties $\cM(\underline{\bd})=\bigcup_{\bv}\cM(\bv,\underline{\bd})$ with extended functions $\sw^{\mathrm{fr}}$, 
where $\underline{\bd}=(\bd_{\In},\bd_{\Out})$.

\begin{Theorem}(Theorem \ref{thm on irr})\label{thm on irr intro}
Fix $\bd_{\In}=-\mu\in \bZ_{\geqslant 0}^{Q_0}$, $\bd_{\Out}=0$ and take $\sw^{\mathrm{fr}}=\sw$ to be the pullback function in either the case of generalized conifolds 
or the polyhedral groups of $D$-type. 

Then $\cH^{\sW^{\mathrm{}}}_{\underline{\bd},\sA^{\mathrm{fr}}}:=\bigoplus_{\bv} H^{\sT\times \sA^{\mathrm{fr}}}(\cM(\bv,\underline{\bd}),\sw^{\mathrm{}})$ is irreducible as a $\mathsf Y_\mu(Q,\sW)$-module over base field $\Frac H_{\sT\times \sA^{\mathrm{fr}}}(\pt)$, where $\sA^{\mathrm{fr}}$ is the framing torus.
\end{Theorem}

In Appendix \ref{sect on appendix A}, we discuss how examples studied above are related to (untwisted) \textit{affine Lie superalgebras} of $A,B,D$ and exceptional types.  
By using similar method in proving Theorem \ref{thm intro SO3}, one can show an injectivity result for specialization maps in affine Lie superalgebras of $C$ type as well (Appendix \ref{sect on appendix B}). Note that deleting the affine vertex yields injectivity of specialization maps for quivers with potentials associated 
to the corresponding finite-dimensional \textit{simple Lie superalgebras}.
In Appendix \ref{sect on im of sp map}, we explicitly calculate the image of an injective (non-isomorphic) specialization map in an example in $D_4$ case. 

Finally we remark that the injectivity results of this paper are expected to hold in critical $K$-theory as well, 
which may be interesting to pursue in the future.

\subsection*{Acknowledgments}

The author is deeply grateful to Andrei Okounkov, Yehao Zhou and Zijun Zhou 
for their collaboration and insightful discussions, without which this work would not have been possible.

\subsection*{Statements and Declarations}
We have no conflicts of interest to disclose.

\subsection*{Declaration on the use of AI} The author used ChatGPT-6 Astra to search references, draw pictures and discuss proof strategy in this work. 
The paper is written by the author, who verified all arguments and references, and takes full responsibility for the content.

\section{Generalities on specialization maps}
 
\subsection{Specialization maps} 

Let $X=R/\!\!/ \sG$ be a smooth quiver variety, with a $\sT$-action and a $\sT$-invariant regular function $\sw$ on it, which descends from a $\sG$-invariant regular function  $\sw: R\to \C$. 
Consider the critical cohomology 
$$H^\sT (X, \sw):=H_\sT^* (X, \varphi_\sw \omega_X), $$
where $\omega_X=\mathbb{D}\mathbb{Q}_X$ is the dualizing sheaf. 
In \cite[Ex.~7.19]{COZZ1}, the authors construct a \textit{specialization map}:
$$\mathsf{sp}: H^{\sT}(X,\sw)\to H^{\sT}(X). $$
It has many remarkable properties, including functoriality with respect to proper pushforward and lci pullback, 
and compatibility with stable envelopes \cite[Props.~7.17, 7.18]{COZZ1}.

Let $\fX=[R/\sG]$ be the stack quotient with induced function $\sw: \fX\to \C$, same construction as above gives a specialization map:
$$\mathsf{sp}: H^{\sT}(\fX,\sw)\to H^{\sT}(\fX). $$

\subsection{Nonabelian stable envelopes} 

The critical cohomologies of $X$ and $\fX$ are related by the \textit{interpolation map}: 
\begin{equation}\label{equ on psi}\bPsi: H^{\sT}(X,\sw)\to H^{\sT}(\fX,\sw), \end{equation}
constructed in \cite[Def.~9.1]{COZZ1} using an idea from \cite{AO1}. 

Let $j: X\to \fX$ be the open inclusion, then 
\begin{equation}\label{psi is inj} j^*\circ \bPsi=\id:  H^{\sT}(X,\sw)\to H^{\sT}(X,\sw). \end{equation}
In particular, $\bPsi$ is an injective map. 

When $X$ is antidominantly framed quiver variety of a symmetric quiver (as in \cite[Thm.~9.20]{COZZ1}), $\bPsi$ is explicitly given by the convolution 
with the  cycle\,\footnote{See the proof of \cite[Lem.~9.21]{COZZ1}.}:
$$\overline{\Delta_{X}}\subset X\times \fX. $$
When $X$ is moreover a symmetric quiver variety, $\bPsi_H$ has a unique characterization 
by degree condition \cite[Thm.~4.18]{COZZ1}, in which case we call it the \textit{nonabelian stable envelope}. 
 
The construction of \eqref{equ on psi} uses existence of stable envelopes with respect to one dimensional torus and 
restriction to open subsets, therefore:
\begin{Proposition}\label{prop on nonabel stab and sp}
We have a commutative diagram: 
\begin{align*}
\xymatrix{
H^{\sT}(X,\sw) \ar[d]_{\bPsi} \ar[r]^-{\mathsf{sp}} & H^{\sT}(X) \ar[d]^{\bPsi} \\
H^{\sT}(\fX,\sw)\ar[r]^{\mathsf{sp}} &  H^{\sT}(\fX).
}
\end{align*}
In particular, if the lower $\mathsf{sp}$ is injective, then the upper $\mathsf{sp}$ is also injective. 
\end{Proposition}

\section{Generalized conifolds}

Generalized conifolds are affine singular varieties:
\begin{equation}\label{gene conifold}\{xy=z^mw^n\}\subseteq \mathbb{C}^4\,, \end{equation}
where $m,n$ are non-negative integers such that $m+n>0$. Their small crepant resolutions form an important class of toric Calabi-Yau 3-folds that do not have 
compact 4-cycles, and have been extensively studied from the perspective of enumerative geometry, e.g.\,\cite{MN, MR1, N,NY, S}, and 
geometric representation theory, e.g.\,\cite{BR, CZ, LY, RSYZ, VV1}.

\subsection{Description by quiver with potential}\label{sect on QP gene conifold}

The resolution is in general not unique and different resolutions are related by flops. One can describe them using symmetric quivers with potentials. 
Fix a map $$\sigma: \{1,2,\ldots,m+n\}\to \{1,-1\}, $$ such that the cardinality of  $\sigma^{-1}(1)$, $\sigma^{-1}(-1)$ is $m$, $n$ respectively.

Let $Q$ be a \textit{quiver} with $(m+n)$ vertices, and arrows given by: 

\begin{itemize}
\setlength{\parskip}{1ex}

\item for any $a\in \mathbb{Z}/(m+n)\mathbb{Z}$, there is an arrow $X_a: a \to a+1$ and an arrow $Y_a:  a+1 \to a$, 

\item there is exactly one loop $\Phi_a: a\to a$ at a vertex $a$ if and only if $\sigma(a)=\sigma(a+1)$,

\item there are no arrows between distinct vertices unless they are adjacent modulo $(m+n)$.


\end{itemize}
For each vertex $a\in \mathbb{Z}/(m+n)\mathbb{Z}$, we fix $\sW_a$ by 

\begin{align*}
\sW_a=
\left\{\begin{array}{rcl}-\sigma(a)\Phi_aY_aX_a+\sigma(a)\Phi_aX_{a-1}Y_{a-1}\,,   & \mathrm{if} \,\,\sigma(a)=\sigma(a+1), \\ 
&  \\
\sigma(a)Y_aX_aX_{a-1}Y_{a-1}\,,\quad\quad \quad\quad &   \mathrm{if} \,\,\sigma(a)=-\sigma(a+1). 
\end{array} \right. 
\end{align*}
Here we read the composition of paths from right to left. 

Then the \textit{potential} of $Q$ we take is the sum: 
$$\sW=\sum_{a\in Q_0}\sW_a.$$
It is elementary to prove the following. 
\begin{Lemma}\label{lem on cut of arrow}
Let 
$$C_+:=\{Y_i\,|\,\sigma(i+1)=1 \}\cup \{\Phi_i\,|\,\sigma(i)=\sigma(i+1)=-1 \}, $$ 
$$C_-:=\{Y_i\,|\,\sigma(i+1)=-1 \}\cup \{\Phi_i\,|\,\sigma(i)=\sigma(i+1)=1 \}. $$
Then every monomial of $\sW$ contains exactly one arrow from $C_+$ and exactly one arrow from $C_-$. 
\end{Lemma}
\begin{Example}
Let $m=2$ and $n=3$, with
\[
  \sigma(1)=\sigma(2)=1,
  \qquad
  \sigma(3)=\sigma(4)=\sigma(5)=-1.
\]
Then 
\[
\begin{aligned}
C_+ = \{Y_1,Y_0,\Phi_3,\Phi_4\}, \quad C_- = \{Y_2,Y_3,Y_4,\Phi_1\},
\end{aligned}
\]
and the quiver with potential is 
\begin{center}
\begin{tikzpicture}[
  scale=0.75,
  transform shape,
  >={Stealth[length=1.8mm]},
  every node/.style={font=\normalsize},
  every edge/.style={draw,->,semithick}
]
  \node (v1) at (0,1.8)    {$1$};
  \node (v2) at (2,3.1)    {$2$};
  \node (v3) at (4,1.8)    {$3$};
  \node (v4) at (3.3,-0.6) {$4$};
  \node (v5) at (0.7,-0.6) {$5\equiv0$};

  \path
    (v1) edge[bend left=14]
         node[above left] {$X_1$} (v2)
    (v2) edge[bend left=14]
         node[below right] {$Y_1$} (v1)

    (v2) edge[bend left=14]
         node[above right] {$X_2$} (v3)
    (v3) edge[bend left=14]
         node[below left] {$Y_2$} (v2)

    (v3) edge[bend left=14]
         node[right] {$X_3$} (v4)
    (v4) edge[bend left=14]
         node[left] {$Y_3$} (v3)

    (v4) edge[bend right=14]
         node[above] {$X_4$} (v5)
    (v5) edge[bend right=14]
         node[below] {$Y_4$} (v4)

    (v5) edge[bend right=14]
         node[right] {$X_0$} (v1)
    (v1) edge[bend right=14]
         node[left] {$Y_0$} (v5);

  \path
    (v1) edge[loop left]
         node[left] {$\Phi_1$} (v1)
    (v3) edge[loop right]
         node[right] {$\Phi_3$} (v3)
    (v4) edge[loop below]
         node[below] {$\Phi_4$} (v4);
\end{tikzpicture}
\end{center}

\[
\begin{aligned}
\sW={}&-{\Phi_1}Y_1X_1
     +{\Phi_1}X_0Y_0
     +Y_2{X_2}X_1Y_1 +\Phi_3Y_3{X_3}
     -\Phi_3{X_2}Y_2 +\Phi_4Y_4{X_4}
     -\Phi_4{X_3}Y_3
     -Y_0X_0{X_4}Y_4.
\end{aligned}
\]
\end{Example}

\subsection{Representations of quivers}

Fix a dimension vector $\bv=(\bv_1,\ldots,\bv_{m+n})\in \mathbb{N}^{Q_0}$ and let $V_i=\C^{\bv_i}$.

The representation space of the quiver $Q$ (without framing) is 
$$R=R(\bv):=\bigoplus_{a\in Q_1}\Hom(\bC^{\mathbf v_{t(a)}},\bC^{\mathbf v_{h(a)}})=P\oplus E_-\oplus E_+, $$
where $t(a)$, $h(a)$ denote the tail, head of an arrow $a\in Q_1$ respectively, 
$P=\bigoplus_i \Hom(V_i,V_{i+1})$ denotes  representations from arrows $X_i$, and $E_{\pm}$ denotes those from arrows in $C_{\pm}$. 

By Lemma \ref{lem on cut of arrow}, the potential may be written as 
\begin{equation}\label{equ on dim red of w}\sw=\tr(\sW)=\sw(X,c_-,c_+)=\langle c_+,s(X,c_-) \rangle, \end{equation}
where $X\in P$, $c_{\pm}\in E_{\pm}$, and 
\begin{equation}\label{equ on sect s} s=s(X,c_-): B:=P\oplus E_-\to E_+^\vee \end{equation}
is linear in $c_-$. We take the gauge group  
$$\sG:=\mathrm{GL}(\bv)=\prod_{i\in Q_0}\mathrm{GL}(\bv_i), $$ 
and two dimensional flavor torus:
\begin{equation}\label{equ on flavor torus conifold} \sT:=\C^*_{h}\times \C^*_{\epsilon}, \end{equation}
which acts on $X_i$, $Y_i$, $\Phi_i$ by $\epsilon$, $\sigma(i+1)h-\epsilon$, $-\sigma(i)h$ respectively. 
It  preserves the potential $\sW$ and $\sw:=\tr\sW$.

\subsection{Dimensional reduction}
Consider the Cartesian diagram 
\begin{align*}
\xymatrix{
Z(s)\times E_+ \ar[d]_{\pi} \ar[r]^-{i} \ar@{}[dr]|{\Box} & R \ar[d]^{\mathrm{projection}} \\
Z(s) \ar[r]^{} &  B\,
}
\end{align*}
Applying \cite[Cor.~A.9]{Dav1} to \eqref{equ on dim red of w}, we have a dimensional reduction isomorphism
$$i_*\circ \pi^*: H^{\sT\times \sG}(Z(s))\cong H^{\sT\times \sG}(R,\sw).  $$

\begin{Lemma}(\cite[Rmk.~7.16]{COZZ1})\label{lem on dim red}
We have a commutative diagram 
\begin{align*}
\xymatrix{
H^{\sT\times \sG}(Z(s)\times E_+) \ar[d]_{i_*}^{\cong} \ar[r]^-{} & H^{\sT\times \sG}(B\times E_+ ) \ar@{=}[d]  \\
H^{\sT\times \sG}(R,\sw)\ar[r]^{\mathsf{sp}} &  H^{\sT\times \sG}(R),
}
\end{align*}
where arrows without symbols denote the pushforward by the natural inclusion.  
\end{Lemma}
Therefore, the study of the specialization map in this case reduces to the study of the pushforward map 
$$H^{\sT\times \sG}(Z(s))\to H^{\sT\times \sG}(B).$$

\subsection{Injectivity of specialization maps after localization in $\C^*_{\epsilon}$}\label{sect on inj sp for gene conifolds}

\begin{Proposition}\label{prop on inj conifold}
The pushforward map 
$$H^{\sT\times \sG}(Z(s))\to H^{\sT\times \sG}(B) $$
is injective after localization with respect to the $\C^*_{\epsilon}$-action.
\end{Proposition}

To prove the proposition, we first consider the following nilpotent locus: 
$$P_{\mathrm{nil}}:=\left\{X=(X_i)_i\in \bigoplus_i \Hom(V_i,V_{i+1})\,\big|\,X_{\Sigma}:=\bigoplus_i X_i: \bigoplus_i V_i\to \bigoplus_i V_i\,\, \mathrm{is}\,\, \mathrm{nilpotent} \right\}, $$ 
$$B_{\mathrm{nil}}:=P_{\mathrm{nil}}\times E_-, \quad Z_{\mathrm{nil}}:=Z(s)\cap B_{\mathrm{nil}}.$$
\begin{Lemma}\label{lem on nilp}
The pushforward along inclusion induces isomorphisms:  
\begin{equation}\label{equ on nilp}H^{\sT\times \sG}(B_{\mathrm{nil}})[\epsilon^{-1}]\cong H^{\sT\times \sG}(B)[\epsilon^{-1}],  \quad 
H^{\sT\times \sG}(Z_{\mathrm{nil}})[\epsilon^{-1}]\cong H^{\sT\times \sG}(Z(s))[\epsilon^{-1}]. \end{equation}
\end{Lemma}
\begin{proof}
For any $1\leqslant k \leqslant \sum_i \bv_i$,  we define $q_k: P\to \C$ by $X\mapsto \tr(X_{\Sigma}^k)$, then 
$$P\setminus P_{\mathrm{nil}}=\bigcup_{k=1}^{\sum_i \bv_i}\big\{q_k\neq 0 \big\}. $$
Here $g\cdot X_i=g_{i+1}X_ig_i^{-1}$, $e^\epsilon\cdot X_i=e^\epsilon X_i$, $e^h\cdot X_i=X_i$, therefore $q_k$ is $\sG$-invariant and has $\C^*_{\epsilon}$-weight $k$.

We claim that for any $(\sT\times \sG)$-scheme $U$, and any function $f: U\to \C^*$ such that 
it is invariant under $(\C^*_{h}\times \sG)$-action and equivariant under $\C^*_{\epsilon}$ with nontrivial weight $k\neq 0$, 
then 
$$\epsilon\cdot H^{\sT\times \sG}(U)=0, $$
which obviously implies that $H^{\sT\times \sG}(U)[\epsilon^{-1}]=0$. To prove the claim, consider $(\sT\times \sG)$-equivariant line bundle 
$$L_k=U\times \C_{k}\to U, \quad (e^h,e^\epsilon,g)\cdot (x,z)=((e^h,e^\epsilon,g)\cdot x,e^{k\epsilon} z). $$
Then $c_1^{\sT\times \sG}(L_k)=p^*(k\cdot \epsilon)$, where $p: U\to \Spec \C$. The function $f$ induces a non-zero 
$(\sT\times \sG)$-equivariant section  
$$s_f: U\to L_k. \quad x\mapsto (x,f(x)). $$
Therefore, $p^*(\epsilon)=0\in H^2_{\sT\times \sG}(U)$. Then for any $\alpha\in H^{\sT\times \sG}(U)$, we have 
$$\epsilon\cdot \alpha=p^*(\epsilon)\cdot \alpha=0. $$
Applying to $U=\{q_k\neq 0 \}\times E_-$ and using excision, the first equation in \eqref{equ on nilp} holds.
The second equation can be similarly deduced, and we omit the details. 
\end{proof}

\begin{proof}[Proof of Proposition \ref{prop on inj conifold}]

By Lemma \ref{lem on nilp}, we are reduce to proving the injectivity of 
$$H^{\sT\times \sG}(Z_{\mathrm{nil}})[\epsilon^{-1}]\to H^{\sT\times \sG}(B_{\mathrm{nil}})[\epsilon^{-1}]. $$
Motivated by \cite[\S 4.2]{MN}, we stratify $P_{\mathrm{nil}}$ by $\sG$-orbits. 
In this case, we have finitely many $\sG$-orbits $\mathcal{O}_1,\ldots, \mathcal{O}_r$ and by suitable relabelling, 
$P_a:=\bigcup_{i\leqslant a}\mathcal{O}_i$ is closed and we have 
$$\emptyset=P_0\subset P_1\subset \cdots \subset P_r=P_{\mathrm{nil}}, \,\,\, \mathrm{with} \,\,P_a\setminus P_{a-1}=\mathcal{O}_a. $$
Note that $\sT$ preserves $\sG$-orbits. In fact, any nilpotent representation in $P$ may be written as a direct sum of graded Jordan chains
of form
\[
 e_0\xrightarrow{p}e_1\xrightarrow{p}\cdots
 \xrightarrow{p}e_{\ell-1}\rightarrow0.
\]
Without loss of generality, we assume $p\in \mathcal{O}_a$. Define a $\C^*$-action by 
\[
 g(t)e_k=t^{-k}e_k.
\]
This preserves every vertex space and hence lies in the gauge group $\sG$, and can cancel the action of $\sT$ on $p$.

By excision, we have a long exact sequence: 
$$\to H_k^{\sT\times \sG}(P_{a-1})\to H_k^{\sT\times \sG}(P_{a})\to H_k^{\sT\times \sG}(\mathcal{O}_{a})\to. $$
Fix $X\in \mathcal{O}_a$, as $\sT\times \sG$ acts on $\mathcal{O}_a$ transitively, we have 
$$[\mathcal{O}_a/(\sT\times \sG)]=[\Spec \C/\mathrm{stab}(X)], $$ 
where $\mathrm{stab}(X)$ is the stabilizer of $X$ under $(\sT\times \sG)$-action. As we are working in $\mathbb{Q}$-coefficient, we have 
\begin{equation}\label{equ on va of odd coho conifolds}H_{\mathrm{odd}}^{\sT\times \sG}(\mathcal{O}_a)=H_{\mathrm{odd}}^{}([\Spec \C/\mathrm{stab}(X)])=0. \end{equation}
Then the long exact sequence breaks into short exact sequence: 
\begin{equation}\label{equ on ses of even coho conifolds}0\to  H_{2k}^{\sT\times \sG}(P_{a-1})\to H_{2k}^{\sT\times \sG}(P_{a})\to H_{2k}^{\sT\times \sG}(\mathcal{O}_{a})\to 0. \end{equation}
By induction, we may reduce to show the injectivity of 
$$j_{a*}: H^{\sT\times \sG}(Z(s)\cap (\mathcal{O}_{a} \times E_-))[\epsilon^{-1}]\to H^{\sT\times \sG}(\mathcal{O}_{a} \times E_-)[\epsilon^{-1}]. $$ 
When $X\in \mathcal{O}_{a}$, the section $s$ \eqref{equ on sect s} has constant rank (because it is $\sG$-equivariant, and linear in $c_-$). 
Denote $K_{\mathcal{O}_{a}}:=\mathrm{Ker}(s(X\in \mathcal{O}_{a},-))$ to be the kernel subbundle of $\mathcal{O}_{a} \times E_-$ on $\mathcal{O}_{a}$, which fits into 
$$0\to K_{\mathcal{O}_{a}}\to \mathcal{O}_{a} \times E_-\to N_{\mathcal{O}_{a}}\to 0, $$
where $N_{\mathcal{O}_{a}}$ denotes the quotient bundle. 
Note that  
$$Z(s)\cap (\mathcal{O}_{a} \times E_-)=\mathrm{Tot}(K_{\mathcal{O}_{a}}), $$
and $j_{a*}$ is the same as pushforward under the inclusion $K_{\mathcal{O}_{a}}\to \mathcal{O}_{a} \times E_-$. 

As $\C^*_{h}$ acts on $P$ trivially (hence also trivially on $\mathcal{O}_{a}$), and acts on $E_-$ by weight $(-1)$, 
the domain of $j_{a*}$ is a polynomial module in $h$ and 
$$j_{a}^*j_{a*}=\cdot e_{\sT\times \sG}(N_{\mathcal{O}_{a}})=(-h)^{\rk}+\mathrm{l.o.t.}.$$
Therefore $\cdot e_{\sT\times \sG}(N_{\mathcal{O}_{a}})$ is injective, so is $j_{a*}$. This finishes the proof.
\end{proof}

\subsection{Purity and integral injectivity}

In this section, we show that the purity of $H^{\sG}(R,\sw)$ follows from Davison's quadratic-deformation theorem \cite{Dav2}. 
Introducing $P_i=Y_iX_i, Q_i=X_{i-1}Y_{i-1}\in \mathrm{End}(V_i)$, then we can write 
$$\sW=\sum_{\sigma(i)=\sigma(i+1)}\sigma(i)\Phi_i(Q_i-P_i)+ \sum_{\sigma(i)\neq \sigma(i+1)}\sigma(i)P_iQ_i.$$
Taking trace gives a function 
$$\sw: R=\bigoplus_i \left(\Hom(V_i,V_{i+1})\oplus \Hom(V_{i+1},V_{i})\right)\oplus \bigoplus_{\sigma(i)=\sigma(i+1)} \Hom(V_{i},V_{i})\to \C. $$
Let $\widetilde{Q}$ be the tripled quiver of the cyclic quiver with $(m+n)$-vertices, i.e. $\widetilde{Q}$ is given by adding one loop to $Q$ at each vertex $i$ such that 
$\sigma(i)\neq \sigma(i+1)$. Denote the loop of $\widetilde{Q}$ at each vertex $i\in Q_0$ by $\Omega_i$, and define 
$$\sW^{\mu}:=\sum_{i\in Q_0}\left(\Omega_i(P_i-Q_i)+\frac{\mu_i}{2}\Omega_i^2\right). $$
where $\mu_i:=(\sigma(i)-\sigma(i+1))/2\in \{0,\pm 1\}$. Taking trace gives a function 
$$\sw^{\mu}:=\tr(\sW^{\mu}): \widetilde{R}:=R\oplus \bigoplus_{\sigma(i)\neq\sigma(i+1)} \Hom(V_{i},V_{i})\to \C. $$
\begin{Lemma}\label{lem on purity of coho gene}
By the $\sG$-equivariant change of variables: 
\begin{align*}
\Omega_i=
\left\{\begin{array}{rcl}-\sigma(i)\Phi_i+\frac{\sigma(i)}{2}(P_i+Q_i)\,,   & \mathrm{if} \,\,\sigma(i)=\sigma(i+1), \\ 
&  \\
\eta_i-\frac{1}{\mu_i}(P_i-Q_i)\,, \quad\quad &   \mathrm{if} \,\,\sigma(i)\neq \sigma(i+1), 
\end{array} \right. 
\end{align*}
we have 
$$\sw^{\mu}=\sw\boxplus \sum_{\sigma(i)\neq\sigma(i+1)}\frac{\mu_i}{2}\tr (\eta_i^2): \widetilde{R}=R\oplus \bigoplus_{\sigma(i)\neq\sigma(i+1)} \Hom(V_{i},V_{i})\to \C. $$
Moreover, the natural mixed Hodge structure on $H^{\sG}(R,\sw)$ is pure. 
\end{Lemma}
\begin{proof}
The equality on functions is by a straightforward calculation.
Use shifted monodromic vanishing cycles with dualizing coefficients, and write
$\mathbb L=\mathbb Q(-1)[-2]$. As the trace pairing on $\End(V_{i})$ is non-degenerate, we have 
$$H^{\sG}\left(\bigoplus_{\sigma(i)\neq\sigma(i+1)} \End(V_{i}),\sum_{\sigma(i)\neq\sigma(i+1)}\frac{\mu_i}{2}\tr (\eta_i^2)\right)\cong H^{\sG}(\pt)\otimes_{\Q} \mathbb L^{-N/2}, 
\,\,\,N=\sum_{\sigma(i)\neq\sigma(i+1)}\bv_i^2. $$
The action of \(\sG\) on this line is trivial, since \(\sG\) is connected. Equivariant Thom-Sebastiani implies
an isomorphism of graded monodromic mixed Hodge structures: 
\[
H^\sG(\widetilde R,\sw^\mu)
\cong H^\sG(R,\sw)\otimes\mathbb L^{-N/2}.
\]
It remains to establish the purity of the left-hand-side. The potential \(\sW^\mu\) is a quadratic deformation of the canonical cubic potential. By \cite[Thm.~A]{Dav2}, its Hodge-theoretic BPS cohomology $\operatorname{BPS}^{\mathrm{hdg}}_{\mu,\mathbf v}$
is pure for every dimension vector \(\mathbf v\).
Let $\mathcal A_\mu^{\mathrm{hdg}}$ denote the critical cohomology of the representation stacks of $(\widetilde Q,\sW^\mu)$, summing over dimension vectors and equipped with the usual virtual normalization. The cohomological integrality theorem \cite{DM} gives an isomorphism of dimension-vector graded monodromic mixed Hodge structures
\[
\mathcal A_\mu^{\mathrm{hdg}}
\cong
\operatorname{Sym}\!\left(
 \bigoplus_{\mathbf v\ne0}
 \operatorname{BPS}^{\mathrm{hdg}}_{\mu,\mathbf v}
 \otimes
 \bigoplus_{j\geqslant 0}\mathbb L^{j+1/2}
\right).
\]
Every factor on the right is pure. Tensor products and graded symmetric powers preserve purity, so each dimension-vector component of $\mathcal A_\mu^{\mathrm{hdg}}$ is pure.
Passing from the virtual normalization to dualizing coefficients makes each component tensor by a power of $\mathbb L^{1/2}$, which also preserves purity. 
So $H^\sG(\widetilde R,\sw^\mu)$ is pure, hence is $H^\sG(R,\sw)$.   
\end{proof}

Combining this with Lemma \ref{lem on dim red} and Proposition \ref{prop on inj conifold}, we have 
\begin{Theorem}\label{thm on sp inj gene stack}
Let $\fX=[R/\sG]$ be a stack of representations of $(Q,\sW)$ associated with a generalized conifold \S \ref{sect on QP gene conifold}. Then the specialization map 
$$\mathsf{sp}: H^{\sT}(\fX,\sw)\to H^{\sT}(\fX)$$
is injective. 
\end{Theorem}

\begin{proof}
There is a Borel spectral sequence (e.g.~\cite[\S 9.2]{Dav3}): 
$$E_2^{p,q}=H^p_{}(B\sT)\otimes H^q_{\sG}(R,\varphi_\sw \omega)\Longrightarrow H^{p+q}_{\sG\times \sT}(R,\varphi_\sw \omega). $$
By Lemma \ref{lem on purity of coho gene}, $E_2^{p,q}$ is pure of weight $p+q$. Every differential raises total degree by one and preserves weights, so all differentials vanish. By induction, the spectral sequence degenerates at $E_2$ page. 

From this, one can get a noncanonical isomorphism of graded $H(B\sT)$-modules
$$H^{\sG\times \sT}(R,\sw)\cong H(B\sT)\otimes H^\sG(R,\sw). $$
In particular, $H^{\sG\times \sT}(R,\sw)$ is free over $H^*(B\sT)=\Q[h,\epsilon]$, and its localization map at $\epsilon$ is injective. Together with Lemma \ref{lem on dim red} and Proposition \ref{prop on inj conifold}, this proves the claimed unlocalized injectivity.
\end{proof}

\subsection{Application to moduli of framed representations}

Fix (framing) dimension vectors 
$$\underline{\bd}=(\bd_{\In},\bd_{\Out})\in \mathbb{N}^{Q_0}\times \mathbb{N}^{Q_0}, $$ 
and the framed representation space: 
\begin{equation}\label{equ rvab}
R(\mathbf{v},\underline{\bd}):=\bigoplus_{a\in Q_1}\Hom(\bC^{\mathbf v_{t(a)}},\bC^{\mathbf v_{h(a)}})\oplus \bigoplus_{i\in Q_0}\left(\Hom(\bC^{\bd_{\mathrm{in},i}},\bC^{\mathbf v_i})\oplus \Hom(\bC^{\mathbf v_i},\bC^{\bd_{\Out,i}})\right).
\end{equation}
There is a projection to the unframed representation space 
$$p: R(\mathbf{v},\underline{\bd})\to R(\mathbf{v}), $$
which is $\sG$-equivariant with respect to the natural action. We extend $\sT$-action to the framed representation space such 
that $p$ is equivariant. Decorate $R(\mathbf{v},\underline{\bd})$ with the pullback potential function by $p$: 
$$\sw:=\tr(\sW): R(\mathbf{v},\underline{\bd})\to \C. $$ 
\begin{Theorem}\label{thm on sp of framed gene conifolds}
We have an injective specialization map for framed representation stack: 
$$\mathsf{sp}: H^{\sT}([R(\mathbf{v},\underline{\bd})/\sG],\sw)\to H^{\sT}([R(\mathbf{v},\underline{\bd})/\sG]). $$
Moreover, for any generic stability condition, we have an injective specialization map for the GIT quotient: 
$$\mathsf{sp}: H^{\sT}(R(\mathbf{v},\underline{\bd})/\!\!/\sG,\sw)\to H^{\sT}(R(\mathbf{v},\underline{\bd})/\!\!/\sG). $$
\end{Theorem}
\begin{proof}
As the function is pullback via $p$, we have a commutative diagram 
\begin{align*}
\xymatrix{
H^{\sT}([R(\mathbf{v},\underline{\bd})/\sG],\sw) \ar[r]^{\mathsf{sp}} & H^{\sT}([R(\mathbf{v},\underline{\bd})/\sG])  \\
H^{\sT}([R(\mathbf{v})/\sG],\sw)\ar[r]^{\mathsf{sp}} \ar[u]^-{p^*}_{\cong} &  H^{\sT}([R(\mathbf{v})/\sG]) \ar[u]^-{p^*}_{\cong}.
}
\end{align*}
Then we apply Theorem \ref{thm on sp inj gene stack} and Proposition \ref{prop on nonabel stab and sp}. 
\end{proof}
\begin{Example}(Hilbert schemes of points on generalized conifolds)

Let $\mathbf{v}=(d,\ldots,d)$, $\bd_{\In}=(1,0,\ldots,0)$, $\bd_{\Out}=(0,\ldots,0)$, Choose stability in the $\mathrm{DT}$-chamber, i.e. 
$$\theta=(-(m+n-1)-\epsilon,1,1,\ldots,1), \quad \mathrm{where} \,\,0<\epsilon\ll 1. $$ 
Then as a subscheme of $R(\mathbf{v},\underline{\bd})/\!\!/\sG$, the critical locus of $\sw$: 
$$\Crit(\sw)\cong \Hilb^d(X) $$
is the Hilbert scheme of $d$-points on $X$. 
Here $X$ is a smooth Calabi-Yau 3-fold which is a crepant resolution of \eqref{gene conifold}. 
In this case, the specialization map is injective. 
\end{Example}

\section{$\mathbb{C}^3/G$ when $G<SO(3,\mathbb{C})$}


For a finite subgroup $G<SL(3,\C)$, there is an associated McKay quiver $Q$ \cite{IN} with a
cubic potential $\sW$ \cite[\S 4.4]{Gin}. Geometrically, there always exists a crepant resolution of the affine quotient: 
$$\widetilde{\C^3/G}\to \C^3/G. $$
By using \cite{BKR}, we know that the derived category of compactly supported sheaves on $\widetilde{\C^3/G}$ is 
equivalent to the derived category of finite dimensional representations of $(Q,\sW)$ \cite[Cor.~4.4.8]{Gin}

To connect to the theory of quantum groups \cite{COZZ1, COZZ2}, we restrict to the case when 
the quiver is symmetric, up to conjugation, they are the cases when $G<SL(2,\mathbb{C})$ or $G<SO(3,\mathbb{C})$ (called \textit{polyhedral groups}). Such finite subgroups have been studied from different perspective, e.g. 
$G$-$\Hilb$ schemes, Craw-Ishii conjecture.

When $G<SL(2,\mathbb{C})$, then $(Q,\sW)$ is the affine $A,D,E$-type tripled quiver with canonical cubic potential, 
the injectivity problem was settled in \cite[Lem.~A.4]{COZZ2}. Below we consider the case of polyhedral groups.

\subsection{Classification of finite subgroups $G<SO(3,\mathbb{C})$} \label{sect on finite subgroup SO(3)}
The classification is given as follows, see e.g.~\cite[\S 2]{NS}: 
\begin{itemize}
\setlength{\parskip}{1ex}

\item ($A$-type) Cyclic group $\mathbb{Z}/n\mathbb{Z}$,

\item ($D$-type) Dihedral group $D_{2n}$,

\item ($E$-type) Tetrahedral group $A_4$, Octahedral group $S_4$, Icosahedral group $A_5$. 

\end{itemize}

For cyclic groups, we are reduced to the affine $A$-type quiver with canonical cubic potential \cite[\S 3.1]{NS} mentioned above. 
For $E$-type, specialization map is not injective in general (even after localization with respect to the flavor torus), 
see \S \ref{app on type E}.
The remainder of this section deals with $D$-type cases.

\subsection{Description by quiver with potential in $D$-type}\label{sect on QP of polyhedra}

Below we use the convention that composition of paths in quivers are written from right to the left.
We follow the notation of \cite[\S 3]{NS}.

\subsubsection{$D_{4}$ case}

\tikzset{
  dq vertex/.style={
    inner sep=2.2pt,
    outer sep=0pt,
    font=\normalsize
  },
  dq arrow/.style={
    -{Latex[length=1.7mm,width=1.2mm]},
    line width=0.55pt,
    shorten >=1pt,
    shorten <=1pt
  },
  dq label/.style={
    font=\small,
    inner sep=1.5pt
  }
}

\newcommand{\DFourQuiver}{%
\begin{tikzpicture}[x=1cm,y=1cm]

  \node[dq vertex] (z)  at (0,2.8)   {$0$};
  \node[dq vertex] (zp) at (-3.4,-2) {$0'$};
  \node[dq vertex] (q)  at (3.4,-2)  {$1$};
  \node[dq vertex] (qp) at (0,0.1)   {$1'$};

  \draw[dq arrow] (z) to[bend right=6]
    node[dq label,left] {$a$} (zp);
  \draw[dq arrow] (zp) to[bend right=6]
    node[dq label,right] {$A$} (z);

  \draw[dq arrow] (z) to[bend left=6]
    node[dq label,right] {$b$} (q);
  \draw[dq arrow] (q) to[bend left=6]
    node[dq label,left] {$B$} (z);

  \draw[dq arrow] (z) to[bend left=15]
    node[dq label,right] {$c$} (qp);
  \draw[dq arrow] (qp) to[bend left=15]
    node[dq label,left] {$C$} (z);

  \draw[dq arrow] (zp) to[bend right=5]
    node[dq label,below] {$d$} (q);
  \draw[dq arrow] (q) to[bend right=5]
    node[dq label,above] {$D$} (zp);

  \draw[dq arrow] (zp) to[bend left=13]
    node[dq label,pos=.65,above,sloped] {$e$} (qp);
  \draw[dq arrow] (qp) to[bend left=13]
    node[dq label,below,sloped] {$E$} (zp);

  \draw[dq arrow] (q) to[bend right=13]
    node[dq label,pos=.65,above,sloped] {$f$} (qp);
  \draw[dq arrow] (qp) to[bend right=13]
    node[dq label,below,sloped] {$F$} (q);

\end{tikzpicture}%
}

\newcommand{\DFourPotential}{%
\begin{aligned}
  \sW={}&Cea+ADb+BFc+Efd-Bda-Cfb-AEc-DFe.
\end{aligned}%
}

\begin{center}
  \scalebox{0.76}{\DFourQuiver}
\end{center}

\begin{equation}\label{equ on d4 quiver potential}
  \DFourPotential
\end{equation}

\subsubsection{$D_{2n}$ when $n=2m$ $(m\geqslant 2)$}\label{sect on even}

${}$ \\

\tikzset{
  dq vertex/.style={inner sep=2.2pt,outer sep=0pt,font=\normalsize},
  dq arrow/.style={-{Latex[length=1.7mm,width=1.2mm]},
                   line width=0.55pt,shorten >=1pt,shorten <=1pt},
  dq label/.style={font=\small,inner sep=1.5pt},
  dq loop/.style={dq arrow,loop above,out=120,in=60,
                  min distance=11mm,looseness=6}
}

 \newcommand{\EvenDihedralQuiver}{%
\begin{tikzpicture}[x=1cm,y=1cm]
  \node[dq vertex] (z)  at (0,1.3)    {$0$};
  \node[dq vertex] (zp) at (0,-1.3)   {$0'$};
  \node[dq vertex] (q1) at (2.15,0)   {$1$};
  \node[dq vertex] (q2) at (4.1,0)    {$2$};
  \node[dq vertex] (qm2) at (7.3,0)  {$m-2$};
  \node[dq vertex] (qm1) at (9.25,0) {$m-1$};
  \node[dq vertex] (r)  at (11.4,1.3)  {$m$};
  \node[dq vertex] (rp) at (11.4,-1.3) {$m'$};

  \draw[dq arrow] (z) to[bend left=13]
    node[dq label,right] {$a$} (zp);
  \draw[dq arrow] (zp) to[bend left=13]
    node[dq label,left] {$A$} (z);
  \draw[dq arrow] (z) to[bend left=12]
    node[dq label,above,sloped] {$c$} (q1);
  \draw[dq arrow] (q1) to[bend left=12]
    node[dq label,below,sloped] {$C$} (z);
  \draw[dq arrow] (zp) to[bend left=12]
    node[dq label,above,sloped] {$d_0$} (q1);
  \draw[dq arrow] (q1) to[bend left=12]
    node[dq label,below,sloped] {$D_0$} (zp);

  \draw[dq arrow] (q1) to[bend left=13]
    node[dq label,above] {$d_1$} (q2);
  \draw[dq arrow] (q2) to[bend left=13]
    node[dq label,below] {$D_1$} (q1);
  \draw[dashed,line width=0.55pt] (q2.east) -- (qm2.west);
  \draw[dq arrow] (qm2) to[bend left=13]
    node[dq label,above] {$d_{m-2}$} (qm1);
  \draw[dq arrow] (qm1) to[bend left=13]
    node[dq label,below] {$D_{m-2}$} (qm2);
  \path (q1) edge[dq loop] node[dq label,above] {$u_1$} (q1)
        (q2) edge[dq loop] node[dq label,above] {$u_2$} (q2)
        (qm2) edge[dq loop] node[dq label,above] {$u_{m-2}$} (qm2)
        (qm1) edge[dq loop] node[dq label,above] {$u_{m-1}$} (qm1);

  \draw[dq arrow] (qm1) to[bend left=12]
    node[dq label,above,sloped] {$C'$} (r);
  \draw[dq arrow] (r) to[bend left=12]
    node[dq label,below,sloped] {$c'$} (qm1);
  \draw[dq arrow] (qm1) to[bend left=12]
    node[dq label,above,sloped] {$B'$} (rp);
  \draw[dq arrow] (rp) to[bend left=12]
    node[dq label,below,sloped] {$b'$} (qm1);
  \draw[dq arrow] (r) to[bend right=13]
    node[dq label,left] {$a'$} (rp);
  \draw[dq arrow] (rp) to[bend right=13]
    node[dq label,right] {$A'$} (r);
\end{tikzpicture}%
}

\newcommand{\EvenDihedralPotential}{%
\begin{aligned}
\sW={}&-Cd_0 a- AD_0c +d_0D_0u_1+cCu_1-\sum_{i=1}^{m-2}D_id_iu_i
   +\sum_{i=2}^{m-1}d_{i-1}D_{i-1}u_i\\
 &-b'B'u_{m-1}-c'C'u_{m-1}-C'b'a'-A'B'c'.
\end{aligned}%
}

\newcommand{\OddDihedralQuiver}{
\begin{tikzpicture}[x=1cm,y=1cm]
  \node[dq vertex] (z)  at (0,1.3)  {$0$};
  \node[dq vertex] (zp) at (0,-1.3) {$0'$};
  \node[dq vertex] (q1) at (2.15,0) {$1$};
  \node[dq vertex] (q2) at (4.1,0)  {$2$};
  \node[dq vertex] (qm2) at (7.3,0) {$m-2$};
  \node[dq vertex] (qm1) at (9.25,0) {$m-1$};
  \node[dq vertex] (qm) at (11.2,0) {$m$};

  \draw[dq arrow] (z) to[bend left=13]
    node[dq label,right] {$a$} (zp);
  \draw[dq arrow] (zp) to[bend left=13]
    node[dq label,left] {$A$} (z);
  \draw[dq arrow] (z) to[bend left=12]
    node[dq label,above,sloped] {$c$} (q1);
  \draw[dq arrow] (q1) to[bend left=12]
    node[dq label,below,sloped] {$C$} (z);
  \draw[dq arrow] (zp) to[bend left=12]
    node[dq label,above,sloped] {$d_0$} (q1);
  \draw[dq arrow] (q1) to[bend left=12]
    node[dq label,below,sloped] {$D_0$} (zp);

  \draw[dq arrow] (q1) to[bend left=13]
    node[dq label,above] {$d_1$} (q2);
  \draw[dq arrow] (q2) to[bend left=13]
    node[dq label,below] {$D_1$} (q1);
  \draw[dashed,line width=0.55pt] (q2.east) -- (qm2.west);
  \draw[dq arrow] (qm2) to[bend left=13]
    node[dq label,above] {$d_{m-2}$} (qm1);
  \draw[dq arrow] (qm1) to[bend left=13]
    node[dq label,below] {$D_{m-2}$} (qm2);
  \draw[dq arrow] (qm1) to[bend left=13]
    node[dq label,above] {$d_{m-1}$} (qm);
  \draw[dq arrow] (qm) to[bend left=13]
    node[dq label,below] {$D_{m-1}$} (qm1);
  \path (q1) edge[dq loop] node[dq label,above] {$u_1$} (q1)
        (q2) edge[dq loop] node[dq label,above] {$u_2$} (q2)
        (qm2) edge[dq loop] node[dq label,above] {$u_{m-2}$} (qm2)
        (qm1) edge[dq loop] node[dq label,above] {$u_{m-1}$} (qm1)
        (qm) edge[dq loop] node[dq label,above] {$u_m$} (qm)
        (qm) edge[dq arrow,loop below,out=240,in=300,
                  min distance=11mm,looseness=6]
             node[dq label,below] {$v$} (qm);
\end{tikzpicture}%
}

\newcommand{\OddDihedralPotential}{%
\begin{aligned}
\sW={}&- C d_0 a-A D_0 c+d_0 D_0 u_1+c C u_1-\sum_{i=1}^{m-1}D_id_i u_i
   +\sum_{i=2}^{m}d_{i-1}D_{i-1}u_i-u_mv^2.
\end{aligned}%
}

\begin{center}
 \scalebox{0.9}{\EvenDihedralQuiver}
\end{center}
\begin{equation}
  \EvenDihedralPotential
\end{equation}

\subsubsection{$D_{2n}$ when $n=2m+1$ $(m\geqslant 1)$}\label{sect on odd}

${}$ \\

\begin{center}
  \scalebox{0.9}{\OddDihedralQuiver}
\end{center}
\begin{equation}
  \OddDihedralPotential
\end{equation}

We separate arrows of the quiver $Q$ into two groups,~e.g.~\cite[Lem.~3.3]{MR2}.
\begin{Lemma}\label{lem on cut of arrow dihedral}
Let 
\begin{align*}
C_+:=
\left\{\begin{array}{rcl}\big\{a, A, a',A',u_1,\ldots, u_{m-1} \big\}\,,   & \mathrm{if} \,\,n=2m,\,\, m\geqslant2, \\ 
&  \\
\big\{a, A,u_1,\ldots, u_{m} \big\}\,, \quad \quad &   \mathrm{if} \,\,n=2m+1,\,\, m\geqslant1. 
\end{array} \right. 
\end{align*}
$$C_-:=\left\{a\in Q_1\,|\,a \notin  C_+\right\}. $$
Then every monomial of $\sW$ contains exactly one arrow from $C_+$ and exactly two arrows from $C_-$. 
\end{Lemma}

\subsection{Representations of quivers}

We first deal with cases except $D_4$. 
Fix a dimension vector $\bv\in \mathbb{N}^{Q_0}$.
The representation space of the quiver $Q$ (without framing) is 
$$R=R(\bv):=\underbrace{\bigoplus_{a\in C_+}\Hom(\bC^{\mathbf v_{t(a)}},\bC^{\mathbf v_{h(a)}})}_{P}\oplus 
\underbrace{\bigoplus_{a\in C_-}\Hom(\bC^{\mathbf v_{t(a)}},\bC^{\mathbf v_{h(a)}})}_{E}, $$
where $t(a)$, $h(a)$ denote the tail, head of an arrow $a\in Q_1$ respectively.

By Lemma \ref{lem on cut of arrow dihedral}, the potential may be written as 
\begin{equation}\sw=\tr(\sW)=\sw(p,e)=\langle p,s(e) \rangle, \nonumber \end{equation}
where $p\in P$, $e\in E_{}$, and 
\begin{equation}\label{equ on sect s dihedral} s=s(e): E\to P^\vee \end{equation}
is quadratic in $e$. We take the gauge group  
$$\sG:=\mathrm{GL}(\bv)=\prod_{i\in Q_0}\mathrm{GL}(\bv_i), $$ 
and one dimensional flavor torus:
\begin{equation}\label{equ on flavor torus Dn}  \sT:=\C^*_{t}, \end{equation}
which acts on arrows in $C_+$, $C_-$ by weight $2$, $-1$ respectively.
It  preserves the potential $\sW$ and $\sw:=\tr\sW$. 

\begin{Remark}\label{rmk on D4}
For the $D_4$ case, we separate arrows in \eqref{equ on d4 quiver potential} into three groups 
$$C_x:=\left\{b,B,e,E\right\}, \quad C_y:=\left\{c,C,d,D\right\}, \quad C_z:=\left\{a,A,f,F\right\}. $$
After fixing a dimension vector $\bv\in \mathbb{N}^{Q_0}$, the representation space 
$$R=R(\bv):=
\underbrace{\bigoplus_{a\in C_x}\Hom(\bC^{\mathbf v_{t(a)}},\bC^{\mathbf v_{h(a)}})}_{X}\oplus 
\underbrace{\bigoplus_{a\in C_y}\Hom(\bC^{\mathbf v_{t(a)}},\bC^{\mathbf v_{h(a)}})}_{Y}\oplus 
\underbrace{\bigoplus_{a\in C_z}\Hom(\bC^{\mathbf v_{t(a)}},\bC^{\mathbf v_{h(a)}})}_{Z}. $$
The potential may be written in the form 
\begin{equation}\sw=\tr(\sW)=\sw(x,y,z)=\tr(x(zy-yz)), \nonumber \end{equation}
where $x\in X$, $y\in Y$, $z\in Z$.  We take the gauge group  
$$\sG:=\prod_{i\in Q_0}\mathrm{GL}(\bv_i), $$ 
and two dimensional flavor group:  
\begin{equation}\label{equ on flavor torus D4}\sT=\{(t_x,t_y,t_z)\in (\C^*)^3\,|\,t_x t_y t_z=1 \}, \end{equation}
which acts on arrows from $C_x, C_y, C_z$ by $t_x, t_y, t_z$ respectively, so it preserves $\sW$ and $\sw$. 
\end{Remark}

\subsection{A base change lemma on vanishing cycles}


In this section, we show a base change property of pushforward of vanishing cycle functors using \cite{DP}. All functors below are derived. 

Let $E$ be a vector space, $B$ be a scheme, and 
$$q: X:=B\times E\to \C $$
be a regular function which is homogeneous of degree $d\geqslant 1$ with respect to $E$ coordinates (no restriction on the dependence on $B$), i.e. for any $t\in \C$, we have
$$q(b,t\cdot e)=t^d q(b,e). $$

\begin{Definition}
Let $\pi: B\times E\to B$ be the projection. For any constructible complex $F\in \mathrm{D}_c^b(B)$, denote 
$$K_q(F):=\pi_*\varphi_q(\pi^*F), $$
where $\varphi_q$ denotes the vanishing cycle functor of $q$. 
\end{Definition}

\begin{Lemma}\label{lem on dist triang}
Let $\pi_1: q^{-1}(1)\to B$ be the projection, then for any $F\in \mathrm{D}_c^b(B)$, there is a distinguished triangle
$$K_q(F)\to F \to \pi_{1*}\pi_{1}^*F\to. $$
\end{Lemma}
\begin{proof}
Denote $\pi_0: X_0:=q^{-1}(0)\to B$ to be the projection with inclusion $i=i_0: X_0\to X$. We have the Milnor triangle: 
$$\varphi_q\to i^*\to \Psi_q\,,$$
where $\Psi_q$ is the nearby cycle functor. Applying $\pi_{0*}$, we obtain 
$$K_q(F)=\pi_{0*}\varphi_q(\pi^*F)\to \pi_{0*}\pi_{0}^*F\to \pi_{0*}\Psi_q(\pi^*F) . $$
Then we are left to show 
\begin{equation}\label{equ to show}\pi_{0*}\pi_{0}^*F\cong F, \quad 
\pi_{0*}\Psi_q(\pi^*F)\cong \pi_{1*}\pi_{1}^*F. \end{equation}
There is a $\C^*$-action on $X=B\times E$ scaling the $E$-direction, i.e. $(b,e)\mapsto (b,t\cdot e)$. 
Denote the open inclusion $$j: X\setminus X_0\to X. $$ 
We know $j^*\pi^*F$ is locally constant on the above $\C^*$-orbits, 
therefore by \cite[Prop.~6.2]{DP}, we have 
$$\pi_*j_!(j^*\pi^*F)=0. $$
Using the distinguished triangle 
$$j_!j^*\to \id \to i_*i^*, $$
we obtain 
\begin{equation*}F\cong \pi_*\pi^*F\cong \pi_*i_*i^*\pi^*F\cong \pi_{0*}\pi_{0}^*F. \end{equation*}
This proves the first isomorphism in \eqref{equ to show}. To show the second, we denote 
$$\rho: \widetilde{X}:=\{(b,e,s)\in X\times \C\,|\,q(b,e)=e^s \}\to X. $$
By the definition of nearby cycle functor, 
$$\Psi_q=i_0^*\rho_*\rho^*. $$
Therefore 
\begin{equation}\label{equ on Fpiqq}\pi_{0*}\Psi_q(\pi^*F)\cong \pi_{0*}i_0^*\rho_*\rho^*\pi^*F\cong \pi_{*}i_{0*}i_0^*\rho_*\rho^*\pi^*F\cong \pi_{*}\rho_*\rho^*\pi^*F. \end{equation}
By \cite[Prop.~6.4]{DP}, we have 
$$\pi_{*}\rho_*\rho^*\pi^*F\cong \pi_{*}i_{1*}i_1^*\pi^*F\cong \pi_{1*}\pi_1^*F. $$
Combining with \eqref{equ on Fpiqq}, we are done.
\end{proof}

\begin{Proposition}\label{base change of van cycle}
For any locally closed subset $i: B'\hookrightarrow B$, and $F\in \mathrm{D}_c^b(B)$, we have 
$$i^!K_q(F)\cong K_{q|_{B'\times E}}(i^!F). $$
\end{Proposition}
\begin{proof}
Taking derivative in $e\in E$, we have $\sum_{j}e_j\frac{\partial(q(b,e))}{\partial e_j}=dq(b,e)$, which is nonzero on $q^{-1}(1)$. 
Therefore $\pi_1: q^{-1}(1)\to B$ is smooth. By base change in the following Cartesian diagram 
\begin{align*}
\xymatrix{
q^{-1}(1)|_{B'} \ar[d]_{\pi'_1} \ar[r]^-{i_1} \ar@{}[dr]|{\Box} & q^{-1}(1) \ar[d]^{\pi_1} \\
B' \ar[r]^{i} &  B\,
}
\end{align*}
we obtain 
\begin{equation}\label{equ on ipi}i^!\pi_{1*}\pi_1^*F\cong \pi'_{1*}i_1^!\pi_1^*F\cong \pi'_{1*}\pi'^*_{1}i^! F. \end{equation}
By Lemma \ref{lem on dist triang}, we obtain 
$$i^!K_q(F)\to i^! F \to i^!\pi_{1*}\pi_{1}^*F, \quad K_{q|_{B'\times E}}(i^!F)\to i^! F \to \pi'_{1*}\pi'^*_{1}i^! F.$$
Comparing these two triangles and using \eqref{equ on ipi}, we are done. 
\end{proof}

\subsection{Injectivity of specialization maps after localization in $\C^*_{t}$ in $D$-type}

\begin{Theorem}\label{prop on inj polyhedral}
The specialization map 
$$\mathsf{sp}: H^{\sT\times \sG}(R,\sw) \to  H^{\sT\times \sG}(R)$$
is injective after localization with respect to the $\C^*_{t}$-action.
\end{Theorem}

\begin{proof}

The strategy is similar as \S \ref{sect on inj sp for gene conifolds}.
Consider the following nilpotent locus:
$$P_{\mathrm{nil}}:=\left\{p=(p_{a})_{a \in C_+}\in \bigoplus_{a\in C_+}\Hom(\bC^{\mathbf v_{t(a)}},\bC^{\mathbf v_{h(a)}})\,\Big|\,p_{\Sigma}:=\bigoplus_{a \in C_+} p_{a}\,\, \mathrm{is}\,\, \mathrm{nilpotent} \right\}. $$ 
By the argument of Lemma \ref{lem on nilp}, where the excision for vanishing cycle cohomology can be applied because of the base change property 
(Proposition \ref{base change of van cycle} applied to $B=P$ and $d=2$), 
one then concludes the commutative diagram 
\begin{align}\label{diag on red to nilp}
\xymatrix{
H^{\sT\times \sG}(P_{\mathrm{nil}}\times E,\sw)[t^{-1}] \ar[d]_{\cong} \ar[r]^-{\mathsf{sp}}   & H^{\sT\times \sG}(P_{\mathrm{nil}}\times E)[t^{-1}] \ar[d]^{\cong} \\
H^{\sT\times \sG}(R,\sw)[t^{-1}] \ar[r]^{\mathsf{sp}} &  H^{\sT\times \sG}(R)[t^{-1}].
}
\end{align}
Write $Q_+$ for the quiver with vertex set $Q_0$ and arrow set $C_+$.
Its connected components are oriented two-cycles, and vertices with a
single loop. Its nilpotent representations are direct sums of graded
Jordan chains. In particular, there are finitely many $\sG$-orbits.
We stratify $P_{\mathrm{nil}}$ by finitely many $\sG$-orbits $\mathcal{O}_a$'s:
$$\emptyset=P_0\subset P_1\subset \cdots \subset P_r=P_{\mathrm{nil}}, \,\,\, \mathrm{with} \,\,P_a\setminus P_{a-1}=\mathcal{O}_a. $$
Note that $\sT$ preserves $\sG$-orbits. In fact, any nilpotent representation $p\in P$ (say $p\in \mathcal{O}_a$) may be written as 
\[
M_p=\bigoplus_{\lambda\in\Lambda}
M_\lambda\otimes\mathbb C^{m_\lambda},
\]
where $M_\lambda$'s are pairwise nonisomorphic indecomposable graded
Jordan chains. On each $M_\lambda$, choose a vertex-homogeneous basis
\[
e_{\lambda,0}\xrightarrow{p}e_{\lambda,1}
\xrightarrow{p}\cdots\xrightarrow{p}
e_{\lambda,\ell_\lambda-1}\rightarrow0.
\]
Define a $\C^*$-action by 
\[
 g(t)(e_{\lambda,k}\otimes v)=t^{-2k}(e_{\lambda,k}\otimes v).
\]
This preserves every vertex space and hence lies in the gauge group $\sG$. A direct calculation shows 
$$g(t)(t^2p)g(t)^{-1}e_{\lambda,k}=e_{\lambda,k+1}, \,\,\, \mathrm{i.e.}\,\,\, g(t)(t^2p)g(t)^{-1}=p. $$
For any $p'=h^{-1}ph \in\mathcal{O}_a$, then 
$$t^2p'=h^{-1}(t^2p) h=h^{-1}(g(t)^{-1}pg(t))h \in\mathcal{O}_a. $$
There is a short exact sequence of the stabilizer groups: 
$$1\to \mathrm{stab}_{\sG}(p)\to \mathrm{stab}_{\sG \times \sT}(p)\to \sT \to 1.$$
We may define a lift $\sT\to \mathrm{stab}_{\sG \times \sT}(p)$, $t\mapsto (g(t),t)$ with image $\widetilde{\sT}$, 
then $\mathrm{stab}_{\sG \times \sT}(p)=\mathrm{stab}_{\sG}(p)\rtimes \widetilde{\sT}$. 

Next we choose a Levi subgroup compatible with this lift. Consider the endomorphism algebra of $M_p$ and its Jacobson radical: 
\[
A_p=\operatorname{End}_{Q_+}(M_p),\quad
J_p=\operatorname{rad}(A_p).
\]
The indecomposable decomposition gives
\[
A_p/J_p\cong
\prod_{\lambda\in\Lambda}\operatorname{End}(\mathbb C^{m_\lambda}).
\]
Consequently,
\[
\mathrm{stab}_{\sG}(p)=U_p\rtimes L,\,\,\,\mathrm{where} \,\,\,
U_p=1+J_p,\,\,\,
L=\prod_{\lambda\in\Lambda}\operatorname{GL}(m_\lambda),
\]
where $L$ acts on the multiplicity spaces $\C^{m_\lambda}$. The group $U_p$ is
unipotent. By construction, $L$ commutes with $g(\mathbb C^*)$.
Moreover, conjugation by $g$ preserves $A_p$ and its radical.
It follows that
\[
\mathrm{stab}_{\sG \times \sT}(p)=U_p\rtimes (L\times\widetilde \sT).
\]
In particular, $L\times\widetilde \sT$ is a Levi subgroup of $\mathrm{stab}_{\sG \times \sT}(p)$.

There is an identification
\begin{equation}\label{equ on two stack quot poly}
[(\mathcal O_a\times E)/(\sG\times\sT)]\cong[E/\mathrm{stab}_{\sG \times \sT}(p)],
\end{equation}
which implies  
$$H_{}^{\sT\times \sG}(\mathcal{O}_{a}\times E,\sw)\cong H_{}^{\mathrm{stab}_{\sG \times \sT}(p)}(E,\sw|_{\{p\}\times E}). $$
Since $\mathrm{stab}_{\sG \times \sT}(p)/(L\times\widetilde \sT)\cong U_p$ is contractible, put $\sw_p(e)=\sw|_{\{p\}\times E}$, we have 
\[
H^{\mathrm{stab}_{\sG \times \sT}(p)}(E,\sw_p)\xrightarrow{\sim}H^{(L\times\widetilde \sT)}(E,\sw_p), \quad
H^{\mathrm{stab}_{\sG \times \sT}(p)}(E)\xrightarrow{\sim}H^{(L\times\widetilde \sT)}(E).
\]
These identifications commute with specialization.
We may therefore compute the orbit specialization map using $L\times\widetilde \sT$.
For an arrow in $C_-$, the component sending a vector in chain
position $k$ to one in chain position $l$ has $\widetilde \sT$-weight $2(k-l)-1$.
Thus all $\widetilde \sT$-weights of $E$ are odd and nonzero. Write
\[
E=E_+\oplus E_-
\]
for the positive and negative weight spaces. Since $\widetilde \sT$
is central in $L\times\widetilde \sT$, both summands are $(L\times\widetilde \sT)$-subrepresentations.

The function $\sw_p$ is an $(L\times\widetilde \sT)$-invariant quadratic form.
Invariance under $\widetilde \sT$ implies that it pairs only opposite
weights. Hence
\[
\sw_p(e_+,e_-)=\langle e_+,s_p(e_-)\rangle
\]
for an $(L\times\widetilde \sT)$-equivariant linear map
$s_p:E_-\to E_+^\vee$. Let $\iota:\ker(s_p)\hookrightarrow E_-$ be the inclusion, with $N_p=E_-/\ker(s_p)$. 
Dimensional reduction and its compatibility
with specialization as in Lemma~\ref{lem on dim red}, identifies
the specialization map on $(E,\sw_p)$ with
\[
\iota_*:H^{L\times\widetilde \sT}(\ker(s_p))\rightarrow H^{L\times\widetilde \sT}(E_-).
\]
As $\ker(s_p)$ is a
vector space, the dimensional reduction shift is even, and
$H^*(B(L\times\widetilde \sT))=H^*(BL)[t]$
is concentrated in even degrees, therefore $H^{L\times\widetilde \sT}(E,\sw_p)$ is concentrated in even degrees. 
The same holds for
$H^{L\times\widetilde \sT}(E)$.

Decompose $N_p$ according to the $\widetilde \sT$-weights:
\[
N_p=\bigoplus_{\nu<0}N_{p,\nu},\quad
r_\nu=\dim N_{p,\nu},\quad r=\dim N_p,
\]
then
$$\iota^*\iota_*=\cdot e_{L\times\widetilde \sT}(N_p)=c\cdot t^{r}+\mathrm{l.o.t.}, \quad c=\prod_{\nu<0}\nu^{r_\nu}\in\mathbb Q^*.$$
Here $t$ is the original flavor parameter, since
$\widetilde \sT\to\sT$ is an isomorphism and $L$ maps trivially to
$\sT$. A polynomial with invertible leading coefficient acts
injectively on $H^*(BL,\mathbb Q)[t]$. Since
$H^{L\times\widetilde \sT}(\ker(s_p))$ is a rank-one free module over this ring, $\iota_*$ is injective, and remains so after inverting $t$.
Thus specialization is injective on every orbit.

Finally, Proposition~\ref{base change of van cycle} gives the
excision sequences for the filtration on $P_a\times E$, compatible
with specialization. The vanishing of odd cohomology on each orbit
implies inductively that the ordinary and critical excision
sequences split into short exact sequences in even degrees.
After localization at $t$, injectivity on the orbit terms and
induction on $a$ give injectivity on $P_{\mathrm{nil}}\times E$.
The localization reduction \eqref{diag on red to nilp} then proves the theorem.
\end{proof}

For the $D_4$ case, the flavor  group is different from other $D_{2n}$ cases, see Remark \ref{rmk on D4}. The argument above applies to this case. 
\begin{Theorem}\label{thm on inj D4}
In the $D_4$ case, the specialization map 
$$\mathsf{sp}: H^{\sT\times \sG}(R,\sw) \to  H^{\sT\times \sG}(R)$$
is injective after localization with respect to the $\C^*_{t_z}$-action.
\end{Theorem}
\begin{proof}
Consider the finite isogeny
\[
\mathbb C_s^*\times\mathbb C_t^*\longrightarrow \sT, \quad
(s,t)\longmapsto
(st^{-1},s^{-1}t^{-1},t^2).
\]
Then the $t$-weights of arrows from $C_x, C_y, C_z$ satisfy: 
\[
\operatorname{wt}_t(C_z)=2,\quad
\operatorname{wt}_t(C_x)=\operatorname{wt}_t(C_y)=-1.
\]The $s$-torus acts trivially on $Z$. Thus the proof of Theorem \ref{prop on inj polyhedral} applies with this additional commuting torus, and inverting the Chern class of $t_z$ is equivalent to inverting
 $t$, since that class pulls back to $2t$.
\end{proof}

As in Theorem \ref{thm on sp of framed gene conifolds}, we have 
\begin{Theorem}\label{thm on framed inn polyhedral}
The injectivity of (localized) specialization map also holds for moduli stacks of (stable) framed representations with function pullback from the stack of unframed representations. 
\end{Theorem}

\subsection{Failure of injectivity when $G<SO(3,\C) $ is $E$-type}\label{app on type E}

We discuss the injectivity problem in the remaining cases of \S \ref{sect on finite subgroup SO(3)}. 
The associated quivers with potentials are drawn below, where $\omega=e^{2\pi \sqrt{-1}/3}$.
 
 \subsubsection{Tetrahedral quiver}

\begin{center}
\resizebox{0.19\linewidth}{!}{%
\begin{tikzpicture}[>=stealth,thick]

  \node (t0) at (0,-2.0) {$0$};
  \node (t3) at (0,0) {$3$};
  \node (t1) at (-1.8,1.7) {$1$};
  \node (t2) at (1.8,1.7) {$2$};

  \draw[->] (t0) to[bend left=16]
    node[left] {$a$} (t3);
  \draw[->] (t3) to[bend left=16]
    node[right] {$A$} (t0);

  \draw[->] (t1) to[bend left=16]
    node[above right] {$b$} (t3);
  \draw[->] (t3) to[bend left=16]
    node[below left] {$B$} (t1);

  \draw[->] (t2) to[bend left=16]
    node[below right] {$c$} (t3);
  \draw[->] (t3) to[bend left=16]
    node[above left] {$C$} (t2);

  \draw[->] (t3)
    to[loop left,out=170,in=220,
       looseness=7,min distance=10mm]
    node[left] {$u$} (t3);

  \draw[->] (t3)
    to[loop right,out=10,in=-40,
       looseness=7,min distance=10mm]
    node[right] {$v$} (t3);

\end{tikzpicture}%
}
\end{center}

\[
\begin{aligned}
\sW=u aA+\omega u bB+\omega^2 u cC-\frac{1}{3}u^3-v aA-\omega^2 v bB-\omega v cC+\frac{1}{3}v^3.
\end{aligned}
\]


\subsubsection{Octahedral quiver}

\begin{center}
\resizebox{0.36\linewidth}{!}{%
\begin{tikzpicture}[>=stealth,thick]

  \node (o0) at (-2.5,0) {$0$};
  \node (o3) at (0,0) {$3$};
  \node (o4) at (2.8,0) {$4$};
  \node (o1) at (5.3,0) {$1$};
  \node (o2) at (1.4,2.0) {$2$};

  \draw[->] (o0) to[bend left=16]
    node[above] {$a$} (o3);
  \draw[->] (o3) to[bend left=16]
    node[below] {$A$} (o0);

  \draw[->] (o3) to[bend left=16]
    node[above left] {$b$} (o2);
  \draw[->] (o2) to[bend left=16]
    node[midway,fill=white,inner sep=1pt] {$B$} (o3);

  \draw[->] (o2) to[bend left=16]
    node[above right] {$c$} (o4);
  \draw[->] (o4) to[bend left=16]
    node[midway,fill=white,inner sep=1pt] {$C$} (o2);

  \draw[->] (o3) to[bend left=16]
    node[above] {$d$} (o4);
  \draw[->] (o4) to[bend left=16]
    node[below] {$D$} (o3);

  \draw[->] (o4) to[bend left=16]
    node[above] {$e$} (o1);
  \draw[->] (o1) to[bend left=16]
    node[below] {$E$} (o4);

  \draw[->] (o3)
    to[loop below,out=-50,in=-130,
       looseness=7,min distance=10mm]
    node[below] {$u$} (o3);

  \draw[->] (o4)
    to[loop below,out=-50,in=-130,
       looseness=7,min distance=10mm]
    node[below] {$v$} (o4);

\end{tikzpicture}%
}
\end{center}

\[
\begin{aligned}
\sW={}u aA-u Bb-u Dd-\frac{1}{3}u^3+v Ee-v cC-v dD+\frac{1}{3}v^3+(\omega^2-\omega)(BCd+cbD).
\end{aligned}
\]

\subsubsection{Icosahedral quiver}

\begin{center}
\resizebox{0.4\linewidth}{!}{%
\begin{tikzpicture}[>=stealth,thick]

  \node (i0) at (-2.5,0) {$0$};
  \node (i1) at (0,0) {$1$};
  \node (i3) at (2.7,0) {$3$};
  \node (i4) at (5.4,0) {$4$};
  \node (i2) at (4.05,2.0) {$2$};

  \draw[->] (i0) to[bend left=16]
    node[above] {$a$} (i1);
  \draw[->] (i1) to[bend left=16]
    node[below] {$A$} (i0);

  \draw[->] (i1) to[bend left=16]
    node[above] {$b$} (i3);
  \draw[->] (i3) to[bend left=16]
    node[below] {$B$} (i1);

  \draw[->] (i3) to[bend left=16]
    node[above left] {$c$} (i2);
  \draw[->] (i2) to[bend left=16]
    node[midway,fill=white,inner sep=1pt] {$C$} (i3);

  \draw[->] (i2) to[bend left=16]
    node[above right] {$d$} (i4);
  \draw[->] (i4) to[bend left=16]
    node[midway,fill=white,inner sep=1pt] {$D$} (i2);

  \draw[->] (i3) to[bend left=16]
    node[above] {$e$} (i4);
  \draw[->] (i4) to[bend left=16]
    node[below] {$E$} (i3);

  \draw[->] (i1)
    to[loop below,out=-50,in=-130,
       looseness=7,min distance=10mm]
    node[below] {$u$} (i1);

  \draw[->] (i3)
    to[loop below,out=-50,in=-130,
       looseness=7,min distance=10mm]
    node[below] {$v$} (i3);

  \draw[->] (i4)
    to[loop below,out=-50,in=-130,
       looseness=7,min distance=10mm]
    node[below] {$w$} (i4);

\end{tikzpicture}%
}
\end{center}

\[
\begin{aligned}
\sW={} -u aA+5u Bb-\frac{2}{3}u^3+15v bB-5v Cc-20v Ee+\frac{10}{3}v^3+5w eE-w dD-\frac{1}{3}w^3-5cEd+10DeC.
\end{aligned}
\]
Fix a dimension vector $\bv\in \mathbb{N}^{Q_0}$, and consider representation space of the above quiver
$$R=R(\bv), $$
with gauge group $\sG$ and flavor group $\sT$. 

The common feature of above cases is the existence of \textit{cubic terms} in \textit{all loop} parameters. As $\sT$ preserves the potential, so its action on all loops are trivial. 

\begin{Proposition}\label{prop on type E}
Notations as above, if there is a loop vertex $i$ such that $\bv_i\neq 0$, then the specialization map 
\begin{equation}\label{equ on sp of E type}\mathsf{sp}: H^{\sT\times \sG}(R,\sw) \to  H^{\sT\times \sG}(R) \end{equation}
is \textit{not} injective after localization with respect to the $\sT$-action.

\end{Proposition}
\begin{proof}
Say the loop vertex $i$ such that $\bv_i\neq 0$ correspond to loop $u$. 
Consider the line in $R$: 
$$l: \mathbb{A}^1\hookrightarrow R, \quad u=l(t)=t\cdot \id_{V_i}. $$
$\sG\times \sT$ acts on the line trivially, and we have 
$$l^*\sw=c \, \bv_i\, t^3, \quad c\neq 0. $$
Note that 
$$\varphi_{t^3} \omega_{\mathbb{A}^1}\cong \C^{\oplus 2}_{\{0\}}[1],  $$
whose cohomology lies in odd degree, so  
$$\mathsf{sp}=0: H^{\sT\times \sG}(\mathbb{A}^1,t^3)\to H^{\sT\times \sG}(\mathbb{A}^1). $$
 By the compatibility with l.c.i.\,pullback, we have a commutative diagram 
\begin{align*}
\xymatrix{
H^{\sT\times \sG}(R,\sw) \ar[d]_{l^!} \ar[r]^-{\mathsf{sp}}   & H^{\sT\times \sG}(R) \ar[d]^{l^!}_{\cong} \\
H^{\sT\times \sG}(\mathbb{A}^1,l^*\sw) \ar[r]^{\,\,\mathsf{sp}} &  H^{\sT\times \sG}(\mathbb{A}^1).
}
\end{align*}
Therefore \eqref{equ on sp of E type} is a zero map. To prove the claimed non-injectivity, we are left to show the domain 
is nonzero after localization. Let $K\subset \sG$ be the maximal torus and let $F:=R^K$ be the torus fixed subspace. 
Then 
$$F\cong \mathbb{A}^r, \quad \mathrm{for}\,\, \mathrm{some}\,\, r>0, $$
which is given by the representation of loops in the quiver. Then 
$$\sw|_{F}=\sum_{a=1}^r c_a x_a^3, \quad c_a\neq 0,$$
and $H^{\sT\times K}(F,\sw|_{F})$ is a free $H^{\sT\times K}(\pt)$-module of rank $2^r$, up to a common grading shift.

Let $\iota: F\to R$ be the inclusion, by the self-intersection formula in critical cohomology, e.g.~\cite[Prop.~B.1]{COZZ1}:
$$\iota^!\circ \iota_*=e_{\sT\times K}(N_{F/R})\neq 0. $$
Hence $\iota_*: H^{\sT\times K}(F,\sw|_{F}) \to H^{\sT\times K}(R,\sw) $ is injective after localization, so the latter is nonzero after localization.

Finally, we have 
$$H^{\sT\times K}(R,\sw)\cong H^{\sT\times \sG}(R,\sw)\otimes_{H^{\sG}(\pt)}H^{K}(\pt). $$
Therefore $H^{\sT\times \sG}(R,\sw)$ is nonzero after localization. 
\end{proof}
\begin{Remark}\label{rmk on type E}
If $\bv_i= 0$ for all loop vertices, then $R=0$ and the specialization map is the identity. 
\end{Remark}

\section{Application to the irreducibility of shifted Yangian modules}

\subsection{Reshetikhin type shifted Yangians via critical stable envelopes}

Let $Q$ be a symmetric quiver with a $\sT$-action, $\sW$ be a $\sT$-invariant potential. Assume that
\begin{itemize}
    \item $\sT$-equivariant weight on every edge loop at $i\in Q_0$ is nonzero.
\end{itemize}
For a given $\mu\in \bZ_{\leqslant 0}^{Q_0}$, there is a \textit{Reshetikhin type} $\mu$-\textit{shifted Yangian}:
$$\mathsf Y_\mu(Q,\sW), $$ 
constructed by geometric $R$-matrices defined by critical stable envelopes \cite[\S 4]{COZZ2}. It naturally acts on the equivariant critical cohomology 
$$\cH^{\sW^{\mathrm{fr}}}_{\underline{\bd},\sA^{\mathrm{fr}}}=H^{\sT\times \sA^{\mathrm{fr}}}(\cM(\underline{\bd}),\sw^{\mathrm{fr}})
=\bigoplus_{\bv} H^{\sT\times \sA^{\mathrm{fr}}}(\cM(\bv,\underline{\bd}),\sw^{\mathrm{fr}})$$
of $\mu=\bd_{\Out}-\bd_{\In}$ antidominantly framed quiver varieties $\cM(\underline{\bd})=\bigcup_{\bv}\cM(\bv,\underline{\bd})$\footnote{Here $\bv$ is the gauge dimension vector and $\underline{\bd}=(\bd_{\In},\bd_{\Out})$ is the framing dimension vector. And we take the cyclic stability condition.} with extended torus $(\sT\times \sA^{\mathrm{fr}})$-action and extended functions $\sw^{\mathrm{fr}}$ invariant under torus action, where $\sA^{\mathrm{fr}}$ is the framing torus.

\subsection{Irreducibility of shifted Yangian modules}

Recall the following criterion of irreducibility of Reshetikhin type shifted Yangian modules. 

\begin{Proposition}(\cite{COZZ2})\label{prop on irr criterion}
Fix $\bd_{\In}=-\mu\in \bZ_{\geqslant 0}^{Q_0}$, $\bd_{\Out}=0$ and take $\sw^{\mathrm{fr}}=\sw$ to be the pullback function. Assume 
\begin{itemize}
\item $\left(\Crit_{R(\bv,\mathbf 0)}(\sw)/\!\!/ \GL(\bv)\right)^{\sT}$ is finite for all $\bv\in \bZ_{\geqslant 0}^{Q_0}$,
\item $\mathsf{sp}_{ }\colon H^{\sT\times \sA^{\mathrm{fr}}}(\cM(\underline{\bd}),\sw^{\mathrm{}})\to H^{\sT\times \sA^{\mathrm{fr}}}(\cM(\underline{\bd}))$ 
  is injective after $\sT$-localization.
\end{itemize}
Then $\cH^{\sW^{\mathrm{}}}_{\underline{\bd},\sA^{\mathrm{fr}}}$ is irreducible as a $\mathsf Y_\mu(Q,\sW)$-module over base field $\Frac H_{\sT\times \sA^{\mathrm{fr}}}(\pt)$.
\end{Proposition}
\begin{proof}
This follows from combining \cite[Rmk.~6.9, Thm.~7.3 and Prop.~7.6]{COZZ2}. 
\end{proof}

We verify assumptions of the above result for quivers with potentials associated with generalized conifolds and polyhedral groups of $D$-type, and 
conclude the following irreducibility result. 

\begin{Theorem}\label{thm on irr}
Fix $\bd_{\In}=-\mu\in \bZ_{\geqslant 0}^{Q_0}$, $\bd_{\Out}=0$ and take $\sw^{\mathrm{fr}}=\sw$ to be the pullback function in either the case of generalized conifolds \S \ref{sect on QP gene conifold} or the polyhedral groups of $D$-type \S \ref{sect on QP of polyhedra}. 

Then $\cH^{\sW^{\mathrm{}}}_{\underline{\bd},\sA^{\mathrm{fr}}}=\bigoplus_{\bv} H^{\sT\times \sA^{\mathrm{fr}}}(\cM(\bv,\underline{\bd}),\sw^{\mathrm{fr}})$ is irreducible as a $\mathsf Y_\mu(Q,\sW)$-module over base field $\Frac H_{\sT\times \sA^{\mathrm{fr}}}(\pt)$.
\end{Theorem}
\begin{proof}
By Proposition \ref{prop on irr criterion} and injectivity of corresponding specialization maps (Theorems \ref{thm on sp of framed gene conifolds}, \ref{thm on framed inn polyhedral}), 
it remains to show the finiteness of torus fixed affine quotients. 
We have an equality of schemes:
\[
\operatorname{Crit}_{R(\mathbf v,0)}(\operatorname{tr}\sW)=\operatorname{Rep}_{\mathbf v}(\operatorname{Jac}(Q,\sW)), 
\]
where the right-hand-side is the representation scheme of the Jacobi algebra of $(Q,\sW)$. 
Indeed, the critical point equations for $\sw=\operatorname{tr}\sW$
are precisely the matrix equations
$\partial_a\sW=0$ for all arrows $a$.
Thus the required finiteness follows from Propositions \ref{lem on finiteness gen conifolds}, \ref{lem on finiteness poly} below.
\end{proof}

\subsection{Finiteness of torus fixed affine quotients}


For an associative and unital algebra $A$ over $\C$, its $d$-dimensional representation scheme has $\C$-points  
\[
\operatorname{Rep}_d(A)(\C)
=
\operatorname{Hom}_{\mathbb C\text{-alg}}\!\bigl(A,M_d(\mathbb C)\bigr),
\]where $M_d(\mathbb C)$ is the algebra of $d\times d$ matrices.
If we have a presentation by generators and relations:
\[
A=\mathbb C\langle x_1,\dots,x_n\rangle/(f_\alpha),
\]a representation is determined by matrices \(X_i\), satisfying the relations of \(A\):
\[
\operatorname{Rep}_d(A)(\C)
=
\left\{(X_1,\dots,X_n)\in M_d(\mathbb C)^n \,|\,
f_\alpha(X_1,\dots,X_n)=0\ \text{for all }\alpha\right\}.
\]
A \emph{central subalgebra} of $A$ is a subalgebra $Z\subseteq Z(A)$ of the center: 
\[
 Z(A)=\{z\in A \,|\, za=az\text{ for all }a\in A\}.
\]

\begin{Lemma}\label{lem on finite number of affine torus fixed}
Let \(A\) be an associative and unital algebra over $\C$ with an algebraic torus \(\sT\)-action. Suppose \(A\) is finite as a module over a \(\sT\)-stable, finitely generated central subalgebra \(Z\) of $A$. If \((\operatorname{Spec}Z)^{\sT}\) is finite, then
\[
\bigl(\operatorname{Rep}_d(A)/\!\!/ \operatorname{GL}(d)\bigr)^{\sT}
\]is finite for every \(d\geqslant 1\).
\end{Lemma}
\begin{proof}
By assumptions, $A$ is a finitely generated $\C$-algebra.
A closed point of the above affine quotient corresponds to a semisimple module: 
\[
M=\bigoplus_{j=1}^r S_j^{\oplus m_j},
\]where \(S_j\) are simple \(A\)-modules.
On each \(S_j\), every \(z\in Z\) acts by an \(A\)-linear endomorphism. Schur's lemma gives
\[
z|_{S_j}=\chi_{S_j}(z)\operatorname{id}_{S_j}.
\] This defines  an algebra homomorphism
\[
\chi_{S_j}:Z\longrightarrow\mathbb C.
\]Its kernel is a maximal ideal \(\mathfrak m_{p_j}\), corresponding to a closed point \(p_j\in\operatorname{Spec}Z\).
Since $[M]$ is $\sT$-fixed, $\{p_j\}_{j=1}^r$ is $\sT$-invariant. Since $\sT$ is
connected, each $p_j$ must be $\sT$-fixed, and belongs to the finite set \((\operatorname{Spec}Z)^{\sT}\).

Since $\mathfrak m_{p_j}$ annihilates $S_j$,  $S_j$ descends to a module of $A_{p_j}:=A/\mathfrak m_{p_j}A$.
Since \(A\) has finitely many generators over \(Z\), their images span \(A_{p_j}\) over \(Z/\mathfrak m_{p_j}=\mathbb C\), hence $A_{p_j}$
is a finite dimensional algebra, which therefore has finitely many simple modules up to isomorphisms. 
By fixing the dimension of representations, we get finitely many torus fixed semisimple isomorphism classes.
The fixed locus is affine of finite type over \(\mathbb C\). Finitely many closed points imply that its coordinate ring is Artinian and finite-dimensional, so the fixed scheme is finite as well.
\end{proof}
\begin{Remark}\label{rmk on dim vect version}
Under the hypotheses of Lemma \ref{lem on finite number of affine torus fixed}, suppose that $1=\sum_{i\in I}e_i$
is a finite decomposition of the identity into pairwise orthogonal idempotents
fixed by $\sT$. For $\mathbf v=(\mathbf v_i)_{i\in I}\in\mathbb N^I$,
let $\operatorname{Rep}_{\mathbf v}(A)$ denote the scheme of
representations on $\bigoplus_{i\in I}\mathbb C^{\bv_i}$ in which
$e_i$ acts as the projection onto $\mathbb C^{\bv_i}$, and put $\sG_{\mathbf v}=\prod_{i\in I}\operatorname{GL}(\bv_i)$.
Then
\[
\bigl(\operatorname{Rep}_{\mathbf v}(A)/\!\!/ \sG_{\mathbf v}\bigr)^\sT
\]
is finite for every $\mathbf v\in \mathbb N^I$.
Indeed, its closed points correspond to $\sT$-fixed isomorphism
classes of semisimple $A$-modules of dimension vector
$\mathbf v$. The proof of Lemma \ref{lem on finite number of affine torus fixed} shows that their simple
constituents belong to a finite list $S_1,\ldots,S_q$.
Their multiplicities $m_j$ satisfy
\[
\sum_{j=1}^q m_j\dim_{\mathbb C}(e_iS_j)=\bv_i
\quad\text{for every }i\in I.
\]
Summing over $i$ bounds every $m_j$, since each $S_j$ has
positive dimension. Hence there are only finitely many
such semisimple modules. The same finite-type argument as in
Lemma~5.3 proves finiteness of the $\sT$-fixed scheme.
\end{Remark}

\begin{Proposition}\label{lem on finiteness gen conifolds}
Let $(Q,\sW)$ be the quiver with potential for generalized conifolds \S \ref{sect on QP gene conifold}, $A=\operatorname{Jac}(Q,\sW)$ the Jacobi algebra, 
and $\sT$ be the flavor torus \eqref{equ on flavor torus conifold}. 
Then $\bigl(\operatorname{Rep}_\bv(A)/\!\!/ \operatorname{GL}(\bv)\bigr)^{\sT}$is finite for every \(\bv\in \mathbb{N}^{Q_0}\).
\end{Proposition}
\begin{proof}
Let
\[
Z=\mathbb C[x,y,z,w]/(xy-z^mw^n).
\]
By \cite[Thm.~1.20]{N}, we have 
\[
A\cong\operatorname{End}_{Y}(\mathcal E),
\]where \(\mathcal E\) is a vector bundle on a resolution
\[
f:Y\longrightarrow X=\operatorname{Spec}Z.
\]
We claim that $Z$ is central in $A$ and $A$ is finite over $Z$. 
For \(r\in Z\), multiplication by \(f^*r\) gives an endomorphism of \(\mathcal E\). Every \(\mathcal O_{Y}\)-linear endomorphism commutes with scalar multiplication, 
this induces a homomorphism 
\[
Z\longrightarrow Z(A). \]
Since \(f\) is dominant and \(X\) is integral, pullback \(Z\to\Gamma(Y,\mathcal O_Y)\) is injective; since \(E\) is locally free of positive rank, scalar multiplication \(\Gamma(Y,\mathcal O_Y)\to\operatorname{End}_Y(E)\) is injective. Their composite is the above homomorphism, which is therefore injective.  
Furthermore, properness of \(f\) implies that $f_*\mathcal End(\mathcal E)$ is coherent on the affine scheme \(X\). Therefore the set $A$ of global sections is a finite \(Z\)-module. 

To show \((\operatorname{Spec}Z)^{\sT}\) is finite, recall the torus weight \eqref{equ on flavor torus conifold}.
Under the toric description \cite[Proof of Prop.~1.16]{N}, the induced weights on the generators of $Z$ are given as follows:  
\(x\) is represented by the full cycle of \(X\)-arrows, giving weight \(N\epsilon\) $(N=m+n>0)$; \(y\) is the full reverse cycle, whose weight is the sum of all \(Y\)-weights, i.e. 
$(m-n)h-N\epsilon$; the representatives of \(z,w\) have weights \(h,-h\) respectively. 
Hence all four characters are nonzero. 

If \(p\in (\operatorname{Spec}Z)\) is $\sT$-fixed and \(f\) is an eigenfunction of nonzero character \(\chi\), then
\[
f(p)=(t\cdot f)(p)=\chi(t)f(p), \,\, \forall\,\,t, \]
forcing \(f(p)=0\). Applying this to the four generators gives
\[
x(p)=y(p)=z(p)=w(p)=0.
\]Therefore $(\operatorname{Spec}Z)^{\sT}=\{0\}$.
Then the claim follows from Lemma \ref{lem on finite number of affine torus fixed} and Remark \ref{rmk on dim vect version}.
\end{proof}

\begin{Proposition}\label{lem on finiteness poly}
Let $(Q,\sW)$ be the quiver with potential for $D$-type subgroups of $SO(3,\C)$ \S \ref{sect on finite subgroup SO(3)}, $A=\operatorname{Jac}(Q,\sW)$ the Jacobi algebra, 
and $\sT$ be the flavor torus \eqref{equ on flavor torus Dn} (and \eqref{equ on flavor torus D4} in $D_4$ case). 
Then $\bigl(\operatorname{Rep}_\bv(A)/\!\!/ \operatorname{GL}(\bv)\bigr)^{\sT}$is finite for every \(\bv\in \mathbb{N}^{Q_0}\).
\end{Proposition}
\begin{proof}
Let 
\[G=D_{2n}< SO(3,\mathbb C),\quad Z=\mathbb C[u,v,w]^G, \quad \C G=\bigoplus_{i\in I}\mathrm{End}(\rho_i), \quad I=\operatorname{Irr}(G), \]
where $\rho_i$ are irreducible complex representations of \(G\). Denote 
$e_i\in \C G$ to be a minimal idempotent in the simple direct summand $\mathrm{End}(\rho_i)$ (so that $\C G\cdot e_i=\rho_i$). Write $e=\sum_{i\in I}^{}e_i$.
By \cite[Thm.~4.4.6 and its proof]{Gin}, we have 
\begin{equation}\label{mckay corr}
A\cong e(\mathbb C[u,v,w]\#G)e,
\end{equation}
where \(\mathbb C[u,v,w]\#G\) is the skew-group algebra. 
Elements in \(\mathbb C[u,v,w]\#G\) are sums
\[
\sum_{g\in G}f_g[g],\quad f_g\in C[u,v,w],
\]and multiplication is defined by
\begin{equation}\label{equ on mult of skew gp}(f[g])(h[k])=f\,g(h)[gk], \end{equation}
where $g(h)$ means polynomial obtained by applying the $g$ action to \(h\), the brackets distinguish the formal group element \([g]\) from its action on polynomials.
The flavor torus $\sT$ commutes with \(G\), acts trivially on \(\mathbb CG\), and therefore fixes \(e\) and the vertex idempotents.
It makes the isomorphism \eqref{mckay corr} $\sT$-equivariant. 

We first claim that \(Z\) is central in \(\mathbb C[u,v,w]\#G\): a $G$-invariant polynomial commutes with all polynomials and with all group elements. Indeed,
for any $z\in \mathbb C[u,v,w]^G$, $g(z)=z$, by \eqref{equ on mult of skew gp}, we have 
\[
[g] z=g(z)[g]=z[g].
\]
The embedding of central subalgebra $Z$ is given by 
\[
Z\hookrightarrow A,\quad z\mapsto ze,
\]
which is injective because \(e\) is a nonzero constant element of \(\mathbb CG\). 

We claim that \(\mathbb C[u,v,w]\) is a finite module over \(\mathbb C[u,v,w]^G\). In fact every \(f\in \mathbb C[u,v,w]\) satisfies the monic polynomial
\[
\prod_{g\in G}(s-g(f))\in \mathbb C[u,v,w]^G[s],
\]hence \(\mathbb C[u,v,w]\) is integral over \(\mathbb C[u,v,w]^G\). Combining with the fact that $\mathbb C[u,v,w]$ is generated by the three coordinates as a \(\mathbb C[u,v,w]^G\)-algebra, so the claim holds. 
Therefore
\[
\mathbb C[u,v,w]\#G=\bigoplus_{g\in G}\mathbb C[u,v,w][g]
\]is also finite over \(Z\). The map
\[
\mathbb C[u,v,w]\#G\rightarrow e(\mathbb C[u,v,w]\#G)e,\quad b\mapsto ebe
\]is \(Z\)-linear and surjective. This proves $A$ is a finite $Z$-module. 

To show \((\operatorname{Spec}Z)^{\sT}\) is finite, recall that the dihedral representation splits into a two-dimensional summand and a one-dimensional summand. 
For $D_{2n}$ with $n>2$, the relevant $\sT$-actions on coordinates are given by
\[
t\cdot u=t^{-1}u,\quad
t\cdot v=t^{-1}v,\quad
t\cdot w=t^2w. \]
Therefore $(\operatorname{Spec}\mathbb C[u,v,w])^{\sT}=\{0\}$.
Since $\sT$ commutes with \(G\), it preserves \(\mathbb C[u,v,w]^G\), so we have a finite, surjective, $\sT$-equivariant quotient map
\[
\pi:\operatorname{Spec}\mathbb C[u,v,w]\longrightarrow\operatorname{Spec}\mathbb C[u,v,w]^G.
\]If \(p\) is a torus-fixed closed point in the target, then \(\pi^{-1}(p)\) is a nonempty finite \(\sT\)-invariant set of closed points. Connectedness of \(\sT\) forces every point in this fiber to be fixed. Since $(\operatorname{Spec}\mathbb C[u,v,w])^{\sT}=\{0\}$, so $p=\pi(0)$. Therefore 
\[
(\operatorname{Spec}\mathbb C[u,v,w]^G)^\sT=\{\pi(0)\}.
\]
For \(D_4\) case, the larger torus is $\sT=\{(t_x,t_y,t_z)\in (\C^*)^3\,|\,t_x t_y t_z=1 \}$. 
It contains the subgroup \(\{(t^{-1},t^{-1},t^2)\,|\,t\in \C^*\}\), so its fixed locus is likewise just the image of the origin.
Then the claim follows from Lemma \ref{lem on finite number of affine torus fixed} and Remark \ref{rmk on dim vect version}.
\end{proof}


\appendix 

\section{Recognition by Lie superalgebras}\label{sect on appendix A}

Based on the physics literature, e.g.~\cite[\S 4.1, Example 1 and Appendix A]{Bao}, quivers with potentials studied in \S \ref{sect on QP gene conifold} and  \S \ref{sect on QP of polyhedra} correspond\footnote{There are different choices of quivers with potentials for the same Lie superalgebra. Different choices are related to each other by Seiberg duality, or 
mutations of quivers with potentials, see e.g.\cite[\S 4.2]{Bao}.} to
(some) untwisted affine superalgebras of type $A,B,D$ and exceptional type: 
\[\begin{aligned}
\{xy=z^mw^n\}\subseteq \mathbb{C}^4
&\longleftrightarrow A(m-1,n-1)^{(1)}, \\
\C^3/D_{4m+2}
&\longleftrightarrow \mathfrak{osp}(3\mid 2m)^{(1)}, \,\,\, m \geqslant 1, 
&&\mathrm{i.e.}\,\, B(1,m)\emph{-type}, \\ 
\C^3/D_{4m}
&\longleftrightarrow \mathfrak{osp}(4\mid 2m)^{(1)},\,\,\, m \geqslant 2,
&&\mathrm{i.e.}\,\, D(2,m)\emph{-type},  \\
\C^3/D_{4} &\longleftrightarrow D(2,1,\alpha)^{(1)}. \\
\end{aligned}\] 
It would be interesting to work out the Lie superalgebras constructed from \cite{COZZ2} and compare them with 
the above affine superalgebras of type $A,B,D$ and exceptional type.

It is unclear to the author which Lie superalgebras, if any, correspond to the $E$-type quivers with potentials associated with the Tetrahedral group $A_4$, Octahedral group $S_4$ and Icosahedral group $A_5$ in $SO(3,\C)$.

\section{Injectivity in the case of $\mathfrak{osp}(2\mid 2m)^{(1)}$ with $m\geqslant 2$}\label{sect on appendix B}

For untwisted affine superalgebras of type $C$, i.e. 
$$\mathfrak{osp}(2\mid 2m)^{(1)}=C(m+1)^{(1)}, \,\,m\geqslant 2, $$ 
their quivers with potentials $(Q,\sW)$ are 
given as below: \\

\newcommand{\OspRank}{0}
\newcommand{\QuiverScale}{0.9}

\tikzset{
  vertex/.style={inner sep=2pt,outer sep=0pt},
  arr/.style={
    -{Latex[length=1.6mm]},
    line width=0.5pt
  },
  lab/.style={font=\small,inner sep=1.5pt},
  adjoint/.style={
    arr,loop above,out=120,in=60,
    min distance=10mm,looseness=6
  }
}

\newcommand{\ArrowPair}[4]{%
  \draw[arr] (#1) to[bend left=14]
    node[lab,above] {$#3$} (#2);
  \draw[arr] (#2) to[bend left=14]
    node[lab,below] {$#4$} (#1);
}

\begin{center}
\scalebox{\QuiverScale}{%
\begin{tikzpicture}[x=1cm,y=1cm]

  \node[vertex] (q0) at (0,1.3)  {$0$};
  \node[vertex] (qp) at (0,-1.3) {$0'$};
  \node[vertex] (q1) at (2.3,0)  {$1$};

  \draw[arr] (q0) to[bend left=13]
    node[lab,right] {$a$} (qp);
  \draw[arr] (qp) to[bend left=13]
    node[lab,left] {$A$} (q0);

  \draw[arr] (q0) to[bend left=13]
    node[lab,above,sloped] {$c$} (q1);
  \draw[arr] (q1) to[bend left=13]
    node[lab,below,sloped] {$C$} (q0);

  \draw[arr] (qp) to[bend left=13]
    node[lab,above,sloped] {$d_0$} (q1);
  \draw[arr] (q1) to[bend left=13]
    node[lab,below,sloped] {$D_0$} (qp);

  \ifnum\OspRank=0\relax

     \node[vertex] (q2)  at (4.5,0) {$2$};
    \node[vertex] (qm1) at (7.2,0) {$m-1$};
    \node[vertex] (qm)  at (9.6,0) {$m$};

    \ArrowPair{q1}{q2}{d_1}{D_1}
    \draw[dashed] (q2.east) -- (qm1.west);
    \ArrowPair{qm1}{qm}{d_{m-1}}{D_{m-1}}

    \foreach \v/\s in {q1/1,q2/2,qm1/{m-1},qm/m} {
      \path (\v) edge[adjoint]
        node[lab,above] {$u_{\s}$} (\v);
    }

  \else

     \foreach \j in {2,...,\OspRank} {
      \pgfmathtruncatemacro{\i}{\j-1}
      \node[vertex] (q\j)
        at ({2.3+2.2*(\j-1)},0) {$\j$};
      \ArrowPair{q\i}{q\j}{d_{\i}}{D_{\i}}
    }

    \foreach \i in {1,...,\OspRank} {
      \path (q\i) edge[adjoint]
        node[lab,above] {$u_{\i}$} (q\i);
    }

  \fi

\end{tikzpicture}%
}
\end{center}

\[
\begin{aligned}
\sW={}&
 \lambda_+ C d_0 a
 +\lambda_- A D_0 c +\mu_0 u_1 cC
 +\mu_1 u_1 d_0D_0 +\sum_{i=1}^{m-2}
 \left(
   \alpha_i u_i D_i d_i
   +\beta_i u_{i+1}d_iD_i
 \right)\\
&+\alpha_{m-1}u_{m-1}^{\,2}D_{m-1}d_{m-1}
 +\beta_{m-1}u_m d_{m-1}D_{m-1}.
\end{aligned}
\]
where $\lambda_\pm$, $\mu_0$, $\mu_1$, $\alpha_i$, $\beta_i\in \mathbb C^*$.

Fix a dimension vector $\bv\in \mathbb{N}^{Q_0}$.
The representation space of the quiver $Q$ (without framing) is 
$$R=R(\bv):=\underbrace{\bigoplus_{a\in C_+}\Hom(\bC^{\mathbf v_{t(a)}},\bC^{\mathbf v_{h(a)}})}_{P}\oplus 
\underbrace{\bigoplus_{a\in C_-}\Hom(\bC^{\mathbf v_{t(a)}},\bC^{\mathbf v_{h(a)}})}_{E}, $$
where $$C_+:=\big\{a, A,u_1,\ldots, u_{m} \big\},\quad C_-:=\left\{a\in Q_1\,|\,a \notin  C_+\right\}. $$
We take the gauge group  
$$\sG:=\prod_{i\in Q_0}\mathrm{GL}(\bv_i), $$ 
and one dimensional flavor torus:
$$ \sT:=\C^*_{t}, $$
which acts on arrows $a, A,u_1,\ldots,u_{m-1}$ by weight $2$, $d_{m-1},D_{m-1}$ by weight $-2$, $u_m$ by weight $4$, all remaining by weight $-1$.
It  preserves the potential $\sW$ and $\sw:=\tr\sW$. 

Although $(Q,\sW)$ above is different from those $D$-type quivers with potentials \S \ref{sect on QP of polyhedra}, we observe that injectivity 
argument in \textit{loc.\,cit.} can be applied to this case.

\begin{Theorem}\label{prop on inj C type super}
Notations as above, the specialization map 
$$\mathsf{sp}: H^{\sT\times \sG}(R,\sw) \to  H^{\sT\times \sG}(R)$$
is injective after localization with respect to the $\sT$-action.
\end{Theorem}

\begin{proof}
The proof of Theorem \ref{prop on inj polyhedral} applies here as it uses the fact that 
the function $\sw$ is quadratic with respect to variables from $C_-$ (and does not use the linearity in variables from $C_+$), which holds in the above setting. 
\end{proof}
\begin{Remark}
The same argument also applies to the case of $\mathfrak{osp}(1\mid 2m)^{(1)}$: \\
\begin{center}
\resizebox{0.53\linewidth}{!}{%
\begin{tikzpicture}[every node/.style={inner sep=2pt}]

\node (q0) at (0,0) {$0$};
\node (q1) at (2.2,0) {$1$};
\node (q2) at (4.4,0) {$2$};
\node (qm1) at (7.3,0) {$m-1$};
\node (qm) at (9.5,0) {$m$};

\draw[->] (q0) to[bend left=16]
  node[above] {$d_0$} (q1);
\draw[->] (q1) to[bend left=16]
  node[below] {$D_0$} (q0);

\draw[->] (q1) to[bend left=16]
  node[above] {$d_1$} (q2);
\draw[->] (q2) to[bend left=16]
  node[below] {$D_1$} (q1);

\draw[dashed] (q2.east) -- (qm1.west);

\draw[->] (qm1) to[bend left=16]
  node[above] {$d_{m-1}$} (qm);
\draw[->] (qm) to[bend left=16]
  node[below] {$D_{m-1}$} (qm1);

 \path[->]
  (q0) edge[loop above,min distance=12mm]
    node[above] {$u_0$} (q0)
  (q1) edge[loop above,min distance=12mm]
    node[above] {$u_1$} (q1)
  (q2) edge[loop above,min distance=12mm]
    node[above] {$u_2$} (q2)
  (qm1) edge[loop above,min distance=12mm]
    node[above] {$u_{m-1}$} (qm1)
  (qm) edge[loop above,min distance=12mm]
    node[above] {$u_m$} (qm);

 \path[->] (qm) edge[loop below,min distance=12mm]
  node[below] {$v$} (qm);

\end{tikzpicture}%
}
\end{center}

 \[
\begin{aligned}
\sW={}&
\alpha_0u_0D_0d_0+\beta_0u_1^2d_0D_0+\sum_{i=1}^{m-1}
\left(
\alpha_i u_iD_id_i
+\beta_i u_{i+1}d_iD_i
\right)
+\gamma u_mv^2.
\end{aligned}
\]
\end{Remark}


\section{An explicit example of specialization maps}\label{sect on im of sp map}

Consider $D_4$ case in \S \ref{sect on QP of polyhedra} with dimension vector $\bv=(1,1,1,0)$. 
Then the quiver with potential reduces to 

\tikzset{
  dq vertex/.style={
    inner sep=2.2pt,
    outer sep=0pt,
    font=\normalsize
  },
  dq arrow/.style={
    -{Latex[length=1mm,width=1mm]},
    line width=0.65pt,
    shorten >=1pt,
    shorten <=1pt
  },
  dq label/.style={
    font=\normalsize,
    inner sep=2pt
  }
}

\newcommand{\DFourQuivers}{%
\begin{tikzpicture}[x=1cm,y=1cm]

  \node[dq vertex] (z)  at (0,1.8)   {$0$};
  \node[dq vertex] (zp) at (-3,-2) {$0'$};
  \node[dq vertex] (q)  at (3,-2)  {$1$};

  \draw[dq arrow] (z) to[bend right=8]
    node[dq label,left] {$a$} (zp);
  \draw[dq arrow] (zp) to[bend right=8]
    node[dq label,right] {$A$} (z);

  \draw[dq arrow] (z) to[bend left=8]
    node[dq label,right] {$b$} (q);
  \draw[dq arrow] (q) to[bend left=8]
    node[dq label,left] {$B$} (z);

  \draw[dq arrow] (zp) to[bend right=7]
    node[dq label,below] {$d$} (q);
  \draw[dq arrow] (q) to[bend right=7]
    node[dq label,above] {$D$} (zp);

\end{tikzpicture}%
}

\newcommand{\DFourPotentials}{%
\begin{aligned}
  \sW={}&ADb-Bda.
\end{aligned}%
}

 \begin{center}
  \scalebox{0.68}{\DFourQuivers}
\end{center}

\begin{equation}
  \DFourPotentials
\end{equation}
The representation space is 
$$R=\Hom(\C,\C)^{\oplus 6}=\mathbb{C}^6,$$
with gauge group  
\[
\sG\cong(\C^*)^3, 
\]
and flavor group
\[
\sT=\{(t_x,t_y,t_z)\in(\C^*)^3:t_xt_yt_z=1\},
\]
which acts on arrows $b,B$ by $t_x$, $d,D$ by $t_y$, and $a,A$ by $t_z$. 

Let $V_i$ denote the
tautological line bundle on $[R/\sG]$, and set
\[
u_i=c_1^{\sG\times \sT}(V_i),\quad
h_x=c_1^{\sT}(\C_x),\quad h_y=c_1^{\sT}(\C_y),\quad
h_z=c_1^{\sT}(\C_z).
\]
where $h_x+h_y+h_z=0$. Then we have 
\begin{equation}\label{eq:ring}
H^{\sG\times \sT}(R)=H^{\sG\times \sT}(\pt)
  =\Q[u_0,u_{0'},u_1,h_y,h_z].
\end{equation}
Introduce
\begin{equation}\label{eq:weights}
\begin{aligned}
p&=\mathrm{wt}_{\sG\times\sT}(d)=u_1-u_{0'}+h_y,&
q&=\mathrm{wt}_{\sG\times\sT}(D)=u_{0'}-u_1+h_y,\\
r&=\mathrm{wt}_{\sG\times\sT}(a)=u_{0'}-u_0+h_z,&
s&=\mathrm{wt}_{\sG\times\sT}(A)=u_0-u_{0'}+h_z.
\end{aligned}
\end{equation}
Then
\begin{equation}\label{eq:polynomial}
H^{\sG\times \sT}(R)\cong \Q[u_0,p,q,r,s].
\end{equation}
The specialization map 
\begin{equation}\label{equ on sp d44}
\mathsf{sp}: H^{\sG\times \sT}(R,\sw) \rightarrow H^{\sG\times \sT}(R),
\end{equation}
is $H^{\sG\times \sT}(R)$-linear so the image $I_{\sw}=\mathrm{Im}(\mathsf{sp})$ is an ideal. 
\begin{Proposition}\label{prop on Isw}
The map \eqref{equ on sp d44} is injective, and 
$$I_{\sw}=(pq,ps,rq,rs)\subset H^{\sG\times \sT}(R)\cong\Q[u_0,p,q,r,s].$$
\end{Proposition}
\begin{proof}
Consider the dimensional reduction data  
\[
B=\C^4_{d,a,D,A},\quad
\sigma=(-ad,AD),\quad
Z=\{ad=0,\ AD=0\}\subset B.
\]
The two components of $\sigma$ have weights $p+r$ and $q+s$.
Therefore $R$ is the total space of the dual of the corresponding
rank-two equivariant bundle on $B$, and
$\sw=\langle(B,b),\sigma\rangle$ is linear in its fibers.

As in Lemma \ref{lem on dim red}, specialization map is identified 
with the proper pushforward for $i:Z\hookrightarrow B$:
\begin{align*}
\xymatrix{
H^{\sG\times \sT}_{}(Z) \ar[d]_{}^{\cong} \ar[r]^-{i_*} & H^{\sG\times \sT}(B) \ar[d]_{}^{\cong}  \\
H^{\sG\times \sT}(R,\sw)\ar[r]^{\mathsf{sp}} &  H^{\sG\times \sT}(R),
}
\end{align*}
where the left vertical map is by dimensional reduction \cite[Appendix A]{Dav1}.
  
There is a product structure 
\[
Z=Z_1\times Z_2, \quad Z_1=\{ad=0\}\subset\C^2_{d,a}, \quad Z_2=\{AD=0\}\subset\C^2_{D,A},
\]
where $\C^2_{d,a}$ has weights $p,r$, and $\C^2_{D,A}$ has weights $q,s$.
By equivariant K\"unneth formula and Lemma \ref{lem:axes}:
$$H^{\sG\times \sT}(B)\cong \Q[p,r]\otimes_{\Q}\Q[q,s]\otimes_{\Q}\Q[u_0], $$
$$H^{\sG\times \sT}(Z)\cong (p,r)\otimes_{\Q} (q,s)\otimes_{\Q}\Q[u_0],$$
and $i_*$ is identified with 
\[
\begin{aligned}
(p,r)\otimes_{\Q} (q,s)\otimes_{\Q}\Q[u_0]
&\rightarrow\Q[p,r]\otimes_{\Q}\Q[q,s]\otimes_{\Q}\Q[u_0], \\
 f\otimes g \otimes h&\mapsto fgh.
\end{aligned}
\]
This is a tensor product of injective maps over the field $\Q$, hence
is injective, with image $(p,r)(q,s)$. 
\end{proof}

\begin{Lemma}\label{lem:axes}
Let $\sT_1=(\C^*)^2$ act on $\C^2_{d,a}$ with weights
$p,r$, and $Z_1=\{ad=0\}$. Then 
\begin{equation}\label{equ on hz1}
H^{\sT_1}(Z_1)
\cong
\frac{\Q[p,r]e_d\oplus\Q[p,r]e_a}
     {\Q[p,r](p e_d-r e_a)},
\end{equation}
where $e_d,e_a$ are fundamental classes of the two axes.
Pushforward to $\C^2_{d,a}$ is injective, with image the ideal $(p,r)$.
\end{Lemma}

\begin{proof}
Write $L_d=\{a=0\}$, $L_a=\{d=0\}$, with $L_d\cap L_a=\{0\}$ and $Z_1=L_d\cup L_a$. The Mayer--Vietoris sequence is 
\begin{equation}\label{MV sequence}
\cdots\rightarrow H_k^{\sT_1}(\{0\})\rightarrow
H_k^{\sT_1}(L_d)\oplus H_k^{\sT_1}(L_a)\rightarrow H_k^{\sT_1}(Z_1)\to \cdots.
\end{equation}
Let $i_d: \{0\}\to L_d$, $i_a: \{0\}\to L_a$ be the inclusion. We have 
$$i_{d*}[0]=p e_d, \quad i_{a*}[0]=r e_a. $$
Their pushforward to $Z_1$ coincide, and the first map in \eqref{MV sequence} is given by $[0]\mapsto p e_d-r e_a$, which is injective.
Since $H_{\mathrm{odd}}^{\sT_1}(\{0\}), H_{\mathrm{odd}}^{\sT_1}(L_d), H_{\mathrm{odd}}^{\sT_1}(L_a)$ all vanish, 
hence $H_{\mathrm{odd}}^{\sT_1}(Z_1)=0$ by the exactness of \eqref{MV sequence}, and we then get \eqref{equ on hz1}. 

The pushforward map along $Z_1\to \C^2_{d,a}$ is $f e_d+g e_a \mapsto fr+gp$.
As $p,r$ are coprime in
$\Q[p,r]$, $fr+gp=0$ implies $f=ph$ and $g=-rh$ for some $h$, which is exactly the relation in the displayed presentation, proving
injectivity.
\end{proof}

Consider a rank-two bundle on $[R/(\sG\times \sT)]$ by the descent of 
\[
E_y:=
\bigl(\Hom(V_{0'},V_1)\otimes\C_y\bigr)
\oplus
\bigl(\Hom(V_1,V_{0'})\otimes\C_y\bigr),
\]
whose $(\sG\times \sT)$-equivariant Chern roots are $p,q$, and hence 
\begin{equation}\label{eq:cy}
c_1^{\sG\times \sT}(E_y)=p+q, \quad c_2^{\sG\times \sT}(E_y)=pq.
\end{equation}
Similarly, the bundle    
\[
E_z:=
\bigl(\Hom(V_0,V_{0'})\otimes\C_z\bigr)
\oplus
\bigl(\Hom(V_{0'},V_0)\otimes\C_z\bigr),
\]
has $(\sG\times \sT)$-equivariant Chern roots $r,s$, giving
\begin{equation}\label{eq:cz}
c_1^{\sG\times \sT}(E_z)=r+s, \quad c_2^{\sG\times \sT}(E_z)=rs.
\end{equation}
By Proposition \ref{prop on Isw}, we have
\begin{Corollary}
$$c_2^{\sG\times \sT}(E_y), \,\, c_2^{\sG\times \sT}(E_z)\in I_\sw, \quad c_1^{\sG\times \sT}(E_y), \, \, c_1^{\sG\times \sT}(E_z)\notin I_\sw. $$
Therefore, $c_2^{\sG\times \sT}(E_y), c_2^{\sG\times \sT}(E_z)$ have unique lifts to $H^{\sG\times \sT}(R,\sw)$, and 
$c_1^{\sG\times \sT}(E_y), c_1^{\sG\times \sT}(E_z)$ do not admit lifts. 
\end{Corollary}

\end{document}